\documentclass[reqno,11pt]{amsart}
\usepackage{amsmath,amssymb,graphicx,amsthm,fullpage,xcolor,mathrsfs}
\usepackage{stmaryrd}
\usepackage[utf8]{inputenc}
\usepackage{xifthen,xcolor,tikz,setspace}
\usetikzlibrary{decorations.pathmorphing,patterns,shapes,calc}

\usepackage[normalem]{ulem}

\usepackage{prettyref}

\usepackage{enumerate}

\usepackage[colorlinks=true]{hyperref}
\colorlet{DarkGreen}{green!50!black}
\colorlet{DarkGray}{gray!60!black}

\DeclareMathAlphabet{\mathcal}{OMS}{cmsy}{m}{n}

\newtheorem{theorem}{Theorem}[section]
\newtheorem*{theorem*}{Theorem}
\newtheorem{lemma}[theorem]{Lemma}
\newtheorem*{lemma*}{Lemma}

\newtheorem{observation}[theorem]{Observation}

\newtheorem{corollary}[theorem]{Corollary}

\theoremstyle{definition}{

\newtheorem{definition}[theorem]{Definition}
\newtheorem*{definition*}{Definition}

\newtheorem{question}{Question}
\newtheorem{remark}[theorem]{Remark}
\newtheorem{step}{Step}
}

\numberwithin{equation}{section}

\newcommand{\cA}{\mathcal{A}}

\newcommand{\cM}{\mathscr{M}}
\newcommand{\cN}{\mathcal{N}}
\newcommand{\cE}{\mathcal{E}}
\newcommand{\cF}{\mathcal{F}}
\newcommand{\cG}{\mathcal{G}}

\newcommand{\sL}{\mathscr{L}}

\newcommand{\cS}{\mathcal{S}}
\newcommand{\bS}{\mathbb{S}}
\newcommand{\sU}{\mathscr{U}}

\newcommand{\R}{\mathbb{R}}

\newcommand{\Z}{\mathbb{Z}}

\newcommand{\prob}{\mathbb{P}}
\newcommand{\E}{\mathbb{E}}
\newcommand{\bE}{\mathbf{E}}

\newcommand{\eps}{\epsilon}

\newcommand{\eqdist}{\stackrel{(d)}{=}}
\newcommand{\convdist}{\stackrel{(d)}{\rightarrow}}
\newcommand{\tensor}{\otimes}

\newcommand{\abs}[1]{\left\lvert#1\right\rvert}
\newcommand{\norm}[1]{\left\lVert#1\right\rVert}

\newcommand{\cov}{\text{Cov}}
\newcommand{\tr}{\operatorname{tr}}
\newcommand{\Ric}{\operatorname{Ric}}
\newcommand{\g}[1]{\left\langle{#1}\right\rangle}

\newcommand{\bn}{\mathbf{n}}
\newcommand{\condI}{\textbf{Condition I}}
\newcommand{\condB}{\textbf{Condition B}}

\renewcommand{\limsup}{\varlimsup}
\newrefformat{prop}{Proposition~\ref{#1}}
\newrefformat{fig}{Fig.~\ref{#1}}
\newrefformat{cor}{Corollary~\ref{#1}}
\begin{document}
\title{Bounding flows for spherical spin glass dynamics}
\author{G\'erard Ben Arous}
\author{Reza Gheissari}
\author{Aukosh Jagannath}

\address[G\'erard Ben Arous]{Courant Institute, New York University and NYU Shanghai}
\email{benarous@cims.nyu.edu}

\address[Reza Gheissari]{Courant Institute, New York University and University of California, Berkeley}
\email{gheissari@berkeley.edu}

\address[Aukosh Jagannath]{Harvard University and University of Waterloo}
\email{a.jagannath@uwaterloo.ca}

\begin{abstract}
We introduce a new approach to studying
spherical spin glass dynamics based on differential inequalities for one-time observables.
Using this approach, we obtain an approximate 
phase diagram for the evolution of the energy $H$ and its gradient under Langevin dynamics for spherical $p$-spin models.
We then derive several consequences of this phase diagram.
For example, at any temperature, uniformly over all starting points, 
the process must reach and remain in an absorbing region of large negative values of $H$
and large (in norm) gradients in order 1 time. 
Furthermore, if the process starts in a neighborhood of a critical point  of $H$ with negative energy,
then both the gradient and energy must \emph{increase} macroscopically under this evolution, 
even if this critical point is a saddle with index of order $N$. {As a key technical tool, we estimate 
Sobolev norms of spin glass Hamiltonians, which are of independent interest.}
\end{abstract}

\maketitle
\vspace{-.6cm}

\section{introduction}\label{sec:intro}
We introduce here a simple new way to study the dynamics of spherical spin glasses that is relevant on short time scales.
We apply this method to the Langevin dynamics for the spherical
$p$-spin model. In a companion work, 
we use this approach to study Langevin dynamics for Tensor PCA and the related signal recovery problem \cite{BAGJ18b}. 

Let  $\cS^N =\mathbb{S}^{N-1}(\sqrt{N})=\{x\in\R^N:\sum x_i^2 = N\},$ denote the $(N-1)$-sphere in dimension $N$ of radius 
 $\sqrt{N}$, and consider the $p$-spin Hamiltonian 
 \begin{equation}\label{eq:p-spin-defn}
 H(x)=\frac{1}{N^{(p-1)/2}}\sum_{i_1,\ldots,i_p=1}^N J_{i_1,\ldots,i_p} x_{i_1}\cdots x_{i_p}\,,
 \end{equation}
where the coefficients, $J_{i_1,\ldots,i_p}$, are i.i.d.\ standard Gaussian random variables. 
Observe that this function is typically of order $\sqrt{N}$, while its extreme values are of order $N$.
Recall that these models are prototypical examples of complex energy landscapes: they have exponentially many 
critical points of every index \cite{ABC13,ABA13,Sub15}. 
We study here the Langevin dynamics corresponding to this Hamiltonian, namely the Markov process $X_{t}$ which is the solution to the stochastic differential equation
\begin{equation}\label{eq:Langevin-p-spin-defn}
\begin{cases}
dX_{t}=\sqrt{2}dB_{t}- \beta\nabla H(X_t)dt\\
X_{0}=x_{0}
\end{cases}
\,,
\end{equation}
where $B_{t}$ is spherical Brownian motion, $\nabla H$ is
the spherical gradient of $H$, and $\beta>0$ is the inverse temperature.
This process is reversible and ergodic with stationary measure given by the \emph{Gibbs measure}, $d\pi_{\beta,N}(x)\propto {e^{-\beta H(x)}}dx$,
where $dx$ is the normalized volume measure on $\cS^N$.

The dynamics of mean-field models of spin glasses has received a tremendous amount of attention in both the physics and mathematics literatures. This history is too rich to be summarized here and we refer the reader to the general surveys \cite{BCKM98,BerBir11,Cug03} in the physics literature and \cite{Gui07,BA02} in the mathematics literature.
One regime of interest in the study of spin glass dynamics are
 exponential timescales at low temperatures, as this 
is the timescale to equilibrium \cite{BAJag17,GJ16}. For 
an overview of related work on spin-glass dynamics on such ``activated" timescales, see \prettyref{sec:history} below.

In this paper, we are interested in the dynamics on much shorter timescales, namely timescales that are order $1$ in $N$. 
The classical approach to studying Langevin dynamics of spherical spin glasses on these timescales is via the ``Cugliandolo--Kurchan equations''.  These govern the evolution of certain natural two-time observables
 via a system of integro-differential equations  \cite{crisanti1993sphericalp, CugKur93, BADG06,BADG01,DGM07}. 
At this time, however, it seems challenging to use these techniques to answer the following kinds of natural questions.

\begin{question}[Going down quickly]
Does the dynamics reach and remain at a macroscopic fraction
of the global minimum energy, i.e., energies of order $N$, in short time? 
\end{question}
\begin{question}[Visiting critical points]
Does the dynamics visit and get slowed by critical points?
\end{question}
\begin{question}[Escaping critical points]
When started in a neighborhood of a critical point, can the dynamics escape in finite time, and,
if so, in which way?
\end{question}
\begin{question}[Varying initial data]
How do the answers to these questions change as we vary the initial energy and  gradient?
\end{question}

Our goal in this paper is to provide an elementary approach to 
obtain precise answers to these  questions. 
As an application of our approach, we obtain the 
following answers to the above. 
\begin{enumerate}
\item There is a constant $T_0$ such that for all initial data, the energy must
reach the extreme, order $N$, scale by time $T_0$, and remain there.
Moreover, we provide explicit bounds on the fraction of the ground state reached depending on $p$ and $\beta$.
\item This absorbing region is a region of macrosopically large gradients (norm of order $\sqrt{N}$). In particular, 
the dynamics does not come close to critical points in order 1 times. 
\item When started from any critical point of $H$, the dynamics must increase instantaneously in gradient. When that critical point has negative energy, the dynamics also \emph{increases} in energy at the order $N$ scale.  Perhaps surprisingly, 
this happens even when the critical point has diverging index. Put simply, 
the dynamics must \emph{climb} high dimensional saddles of negative energy for an order 1 time.
\item One can obtain sharper results for any fixed energy and norm of gradient 
of the initial data.
\end{enumerate}

The core of our approach is what we call \emph{bounding flows}. 
We analyze the evolution  of
\begin{equation}\label{eq:U-def}
U_N(t)= (u_N(t),v_N(t))=\bigg(-\frac{H(X_{t})}{N},\frac{\abs{\nabla H(X_t)}^{2}}{N}\bigg),
\end{equation}
from any initial point,
for any finite $T$. We do this by studying all possible weak limit points of the laws of $U_N$ in 
the space of probability measures, $\cM_1(C([0,T])^2)$ .
To this end, we develop an elementary, abstract approach to  show that 
these limit points exist and are $C^1$, by reducing this tightness,
to estimating a Sobolev norm of $H$, which we call the $\cG$-norm {(see Definition~\ref{def:g-norm})}.
This norm is the natural norm in which to study high dimensional problems of this type
and {plays an important role in the accompanying~\cite{BAGJ18b}}. We estimate it using tools from Gaussian comparison theory and differential geometry in Section~\ref{sec:regularity}.
We then show that these observables form what we call a \emph{quasi-autonomous} family. 
Roughly, this is a system of observables, which is closed but for an additional set of ``auxiliary observables"
which themselves have limits. (For a precise definition see \prettyref{sec:quasiautonomy}.) 
In the case of the variables $u_N,v_N$, there is a single auxiliary observable which is related to the alignment of the gradient
of $H$
with the eigenvectors of its Eucildean Hessian along the flow.
We then {uniformly bound the auxiliary observable by }``truncating" the Euclidean Hessian, yielding a pair of  differential inequalities  for $\dot U$ only
in terms of the pair $(u,v)$. This allows us to confine our dynamics 
for any choice of initial energy and gradient by the flowlines for the corresponding autonomous systems.
In this paper, our truncation is elementary and one can imagine continuing this program
by treating the Hessian as a part of our ``quasi-autonomous" family. This would yield new auxilliary 
variables and thereby sharpen the results. 

\subsection{Bounding flows}\label{sec:b-flows-intro}
\begin{figure}[t]
\begin{tikzpicture}
\node at (0,0) {\includegraphics[width=0.7\textwidth]{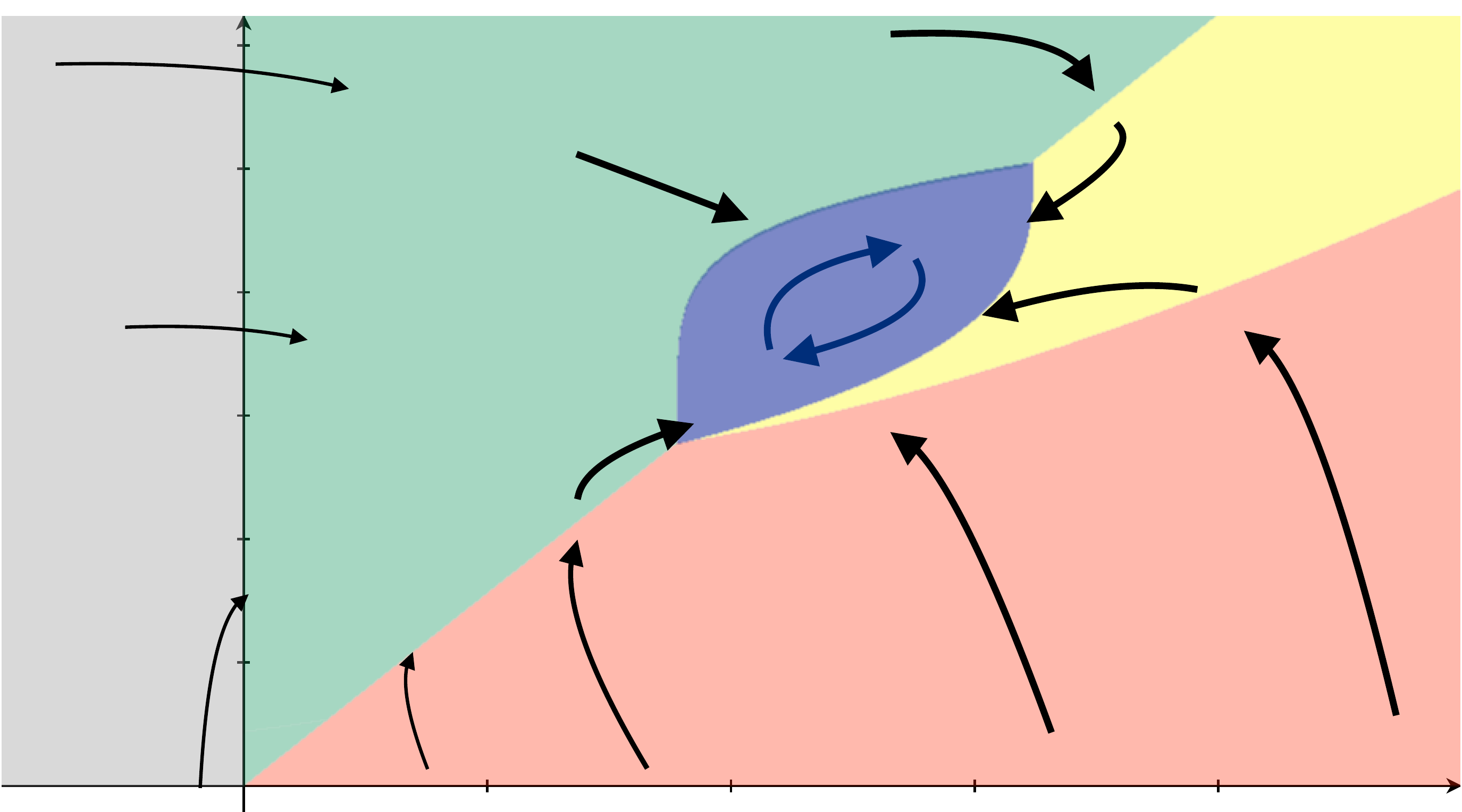}};
\node [font=\Large] at (4.25,-2.65) {$\mathbf{-\frac HN}$};
\node [font=\Large] at (-3,3) {$\mathbf{\frac{|\nabla H|^2}N}$};
\end{tikzpicture}
\vspace{-.1cm}
\caption{A schematic of our approximate phase diagram for $U_N(t)$. Arrows indicate from and to what regions $U_N(t)$ can travel; in particular, the dark blue set is an absorbing set which $U_N(t)$ reaches in finite time, uniformly over starting positions. See Theorem~\ref{thm:full-phase-portrait} for a formal statement of the above depiction. }\label{fig:approx-pd} 
\end{figure}

We begin by proving the following differential inequalities for the evolution of $U$ in Section~\ref{sec:diff-ineq} (see Section~\ref{subsec:derivation-diff-ineq} for a sketch of the derivation). In the following we equip $C([0,T])^k$ with the
norm topology.

\begin{theorem}\label{thm:bounding-flows-thm}
For every $p,\beta,T$ and every sequence of initial points $x_N\in\mathcal S^N$, the laws of $U_N(t)$ from \eqref{eq:U-def} are tight in  
$\cM_1(C([0,T])^2)$. Moreover, every limit $U(t)=(u(t),v(t))$ is almost surely continuously differentiable and satisfies 
\begin{equation}\label{eq:diff-ineq-main}
\dot u = \cF_1 (u,v) \qquad \mbox{and}\qquad  \cF_2^L(u,v)\leq \dot v\leq \cF_2^U(u,v)\,,
\end{equation}
 where, for some $\Lambda_p<\infty$ (see Remark~\ref{rem:lambda-u-c} for explicit bounds),
\begin{equation}\label{eq:cf-def}
\begin{aligned}
\cF_1(u,v)& = -pu+\beta v\,, \\
\cF_2^L (u,v)& = 2p(p-1)- 2(p-1)v + 2pu(pu-\beta v)- 2\beta \Lambda_p v\,, \\
\cF_2^U (u,v)& = 2p(p-1)- 2(p-1)v + 2pu(pu-\beta v)+2\beta \Lambda_p v\,. 
\end{aligned}
\end{equation}
\end{theorem}
%
More broadly, this enables us to bound the trajectory of $U_N(t)$ as we vary the initial data. 
The differential inequalities above translate to precise inequalities for the integral curves of the corresponding dynamical systems. Indeed,
for any initial data $U(0)$, the integral curves of $U(t)$ are almost surely confined to the region in between the 
corresponding integral curves of the lower and upper bounding differential systems of Theorem~\ref{thm:bounding-flows-thm} 
until they hit the stationary line for the energy
(this is stated precisely in \prettyref{cor:graph-cor}). 
To see this by way of example, the right frame of Fig.~\ref{fig:bounding-flow-lines} illustrates how the flows confine the dynamics 
started from a point that is chosen with respect to the uniform measure on $\cS^N$, as well as started from a near-critical point, to
narrow channels that reach the absorbing region.  In 
particular, the confining regions
become sharper near critical points and as well as when $\beta$ gets small. Using these arguments, we obtain the phase diagram depicted in Figure~\ref{fig:approx-pd} on order $1$ timescales by comparing any limiting $U(t)$ to the bounding flows. The various regions of Figure~\ref{fig:approx-pd} are explicitly defined in Section~\ref{subsec:phase-portrait-defs}, and the phase diagram indicated by the arrows in the figure is stated precisely in Theorem~\ref{thm:full-phase-portrait}.

\subsection{Answers to the four questions}\label{sec:p-spin-intro}
\begin{figure}[t]
\vspace{-.5cm}
\resizebox{16cm}{7cm}{
\begin{tikzpicture}
\centering
\node at (-3,0) {\includegraphics[width=0.44\textwidth]{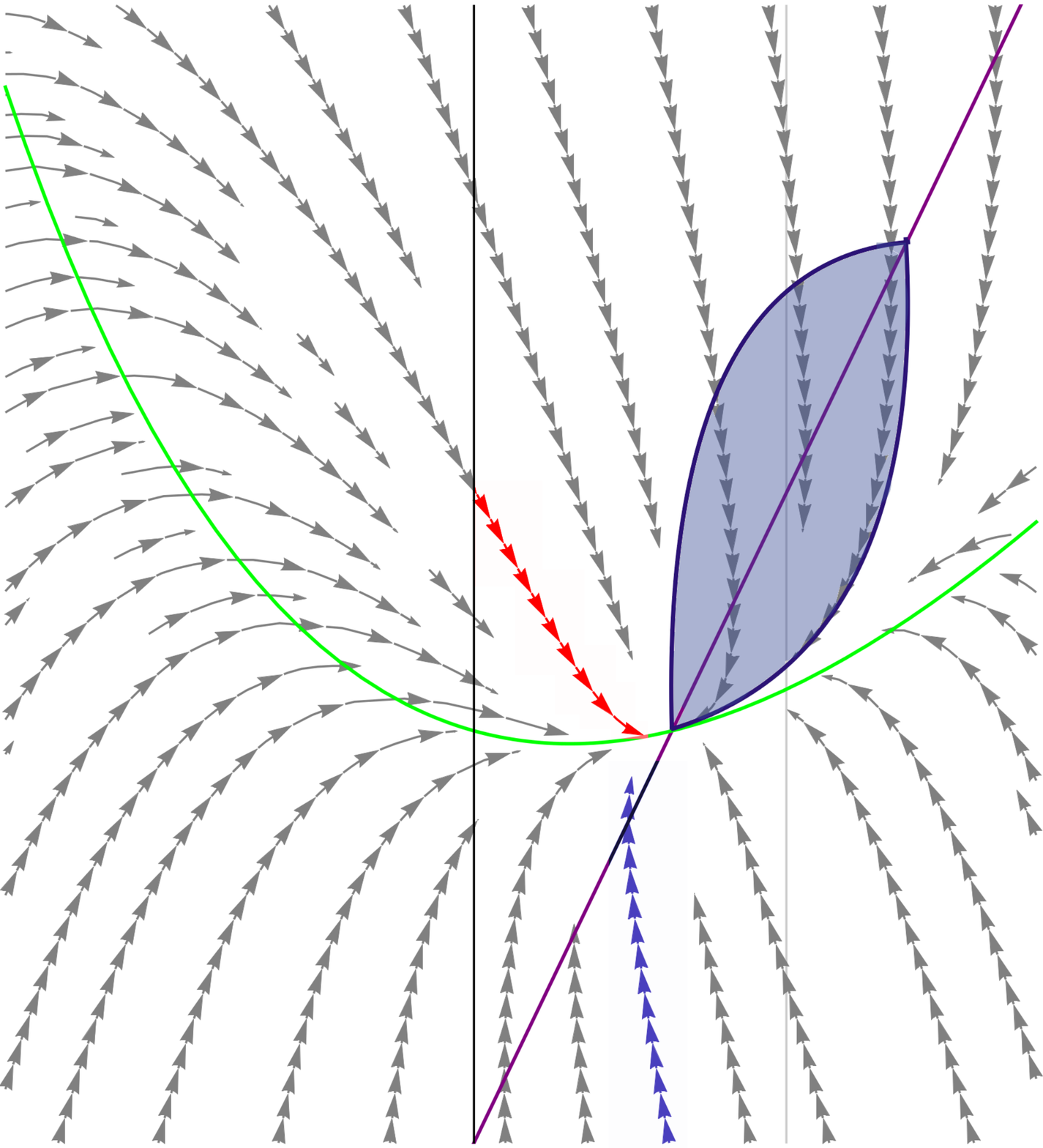}};
\node at (5,0) {\includegraphics[width=0.49\textwidth]{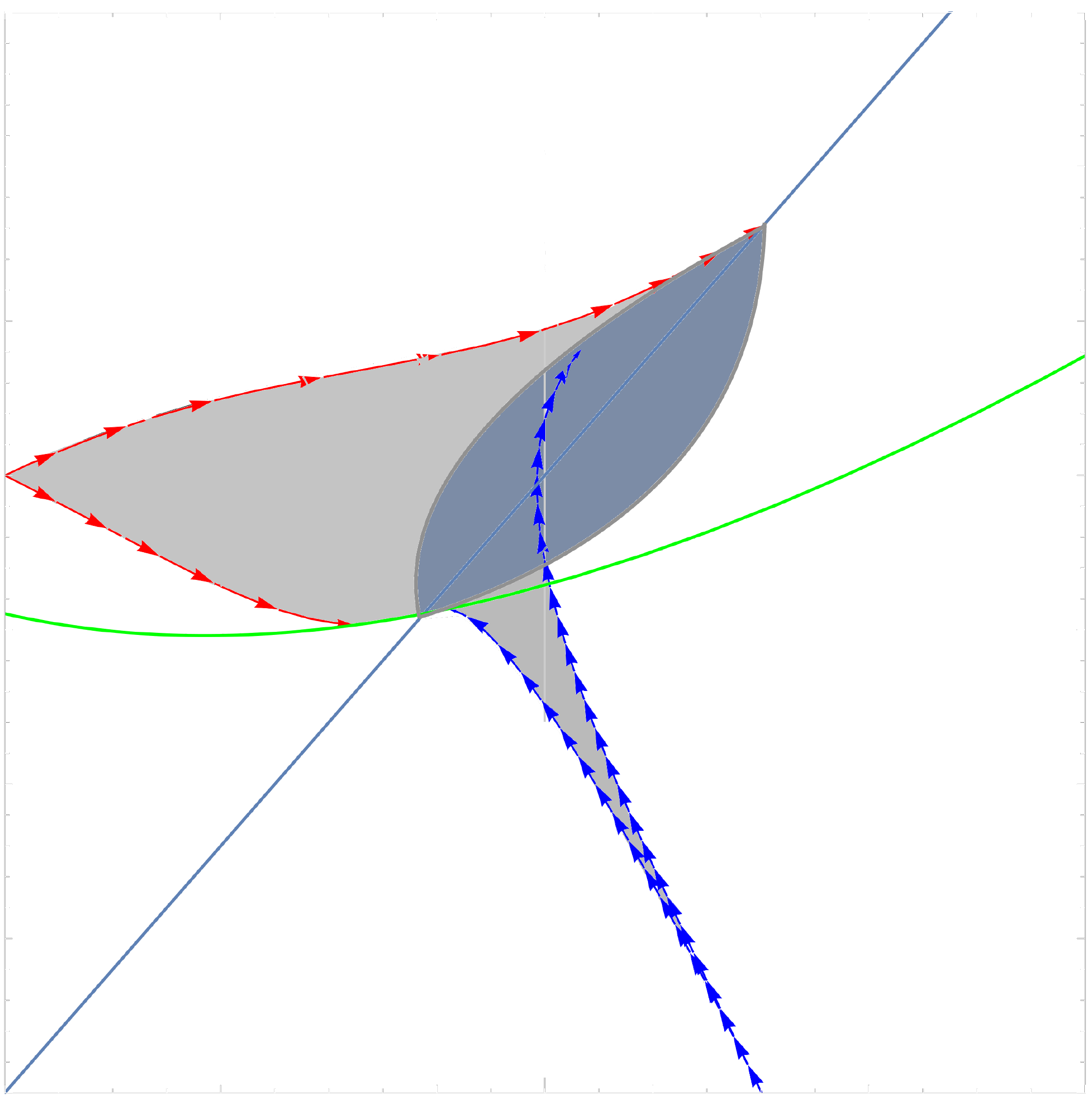}};
\draw [style=thick] (-6.75,-4)--(.75,-4);
\draw [style=thick] (1.07,-4)--(8.93,-4);
\node  at (.3,-3.7) {$\mathbf{-\frac HN}$};
\node  at (-3.8,3.7) {$\mathbf{\frac{|\nabla H|^2}N}$};
\node  at (8.6,-3.7) {$\mathbf{-\frac HN}$};
\node  at (1.5,3.7) {$\mathbf{\frac{|\nabla H|^2}N}$};
\end{tikzpicture}
}
\vspace{-.6cm}
\caption{Left: Example flowlines for the ``lower bounding flow" $\dot A_L(t)= (\mathcal F_1, \mathcal F_2^-)$. The red flowline indicates a start from a point chosen uniformly on $\mathcal S^N$, and the blue flowline is starting from a critical point of negative energy. The blue shaded region is the absorbing set which the process reaches in finite time. Right: $U$ started from a uniformly chosen point is confined to the shaded region between the red curves as it approaches the absorbing region.
Similarly, started from a critical point of negative energy,  $U$ is constrained to the shaded region between the blue curves.}
\label{fig:bounding-flow-lines}
\end{figure}
To understand the consequences of  Theorem~\ref{thm:bounding-flows-thm}, let us now state
more precisely the answers to the four questions from \prettyref{sec:intro}.
We present the full phase diagram from \prettyref{fig:approx-pd} in Section~\ref{sec:outline-and-ideas} and prove it in \prettyref{sec:phase-portrait}.

Throughout the paper, let $\mathbb P$ denote the law of $H$ and $Q_x$ denote the law of $X_t$ started from $X_0= x$. As an answer to Questions 1 and 2, we obtain the following.
\begin{theorem}[Going down quickly and avoiding critical points]\label{thm:performance-guarantee}
For every $\beta>0$ and $p>2$, there exists an explicit $u_c>0$ such that the following holds. For every $\epsilon>0$, there exists $T_0(\epsilon)>0$ such that $\prob$--almost surely,
\begin{equation}\label{eq:performance-guarantee}
\lim_{N\to\infty}\inf_{x\in\cS^N}Q_x\left(\left\{ 
\begin{array}{l} H(X_t)\leq -(u_c-\eps)N \\ \abs{\nabla H(X_t)}^2>(\frac{pu_c}{\beta}-\epsilon)N \end{array}
\text{ for all } t\in[T_0,T]\right\}\right) = 1\,,
\end{equation}
for every $T> T_0$.
\end{theorem}
Before discussing the remaining questions, we pause to make the following important remark  
on the constants $\Lambda_p,u_c$ and the dependence of $u_c$ on $\beta$.

\begin{remark}\label{rem:lambda-u-c}
 Let $G$ be the restriction of
the Euclidean Hessian of $H$ to the tangent space, $T_x \cS^N$ (see \eqref{eq:g-def}); then $\Lambda_p$ in \eqref{eq:cf-def} can be taken to be any constant satisfying
\begin{equation}\label{eq:Lambda-p-defn}
\Lambda_p > \limsup_{N\to\infty} \mathbb E \Big[\mbox{$\sup_{x\in \mathcal S^N}$} \abs{G(x)}_{op}\Big]\,,
\end{equation}
where $\abs{G}_{op}$ denotes the operator norm  for $G$ when viewed as a bilinear form on $T_x \cS^N$.
An explicit, but suboptimal, admissible choice of  $\Lambda_p$ is 
\begin{align}\label{eq:choice-of-Lambdap}
\Lambda_p = \sqrt{p(p-1)}(\sqrt{2}+E_{0,p-2})\,,\qquad \mbox{where}\qquad E_{0,p}:=\lim_{N\to\infty} \big\|{\frac{H_{N,p}}{N}}\big\|_{L^{\infty}}\,,
\end{align}
seen to satisfy~\eqref{eq:Lambda-p-defn} via standard Gaussian comparison inequalities~\cite{GJ16}. 
(For $p=1,2$ we have $E_{0,1}=1$ and $E_{0,2}=\sqrt 2$; for $p\geq 3$, see \cite{ABA13,JagTobLT16,ChenSen15} for explicit formulae for the ground state energy). 

For any choice of $\Lambda_p$ and $\beta>0$, the minimal energy reached is bounded above by $-u_c N$ where 
\begin{align}
u_c = u_c(\beta) & =  \beta\left(1 +  \tfrac{\beta \Lambda_p}{p-1}\right)^{-1}\,. \label{eq:u_c}
\end{align}
Using the sub-optimal choice of $\Lambda_p$ from~\eqref{eq:choice-of-Lambdap}, we can plug in for various values of $p$, and obtain explicit upper bounds on the energy the process reaches and remains at for all order one times. In particular, if we compare the low-temperature values of $-u_c N$ to the conjectured \emph{threshold energy} $-E_{\infty,p} N$  (the energy at which low-temperature Langevin dynamics is believed to get stuck on order one times when initialized uniformly at random) the choice of $\Lambda_p$ from~\eqref{eq:choice-of-Lambdap} ensures that uniformly over all initializations, the energy reaches the following fraction of the threshold in order 1 times:   
\begin{align*}
\sup_{\beta} \frac{u_c(\beta)}{E_{\infty,p}} = \frac{1}{2(\sqrt 2+E_{0,p-2})}\,.
\end{align*}
We plot the resulting lower bounds on $u_c/E_{\infty,p}$ as a function of $\beta$ for various choices of $p$ in Figure~\ref{fig:quantitative-uc}. 
\end{remark}

\begin{figure}[t]
\vspace{-.5cm}
\begin{tikzpicture}
\centering
\node at (-4.1,-.05) {\includegraphics[width=0.445\textwidth]{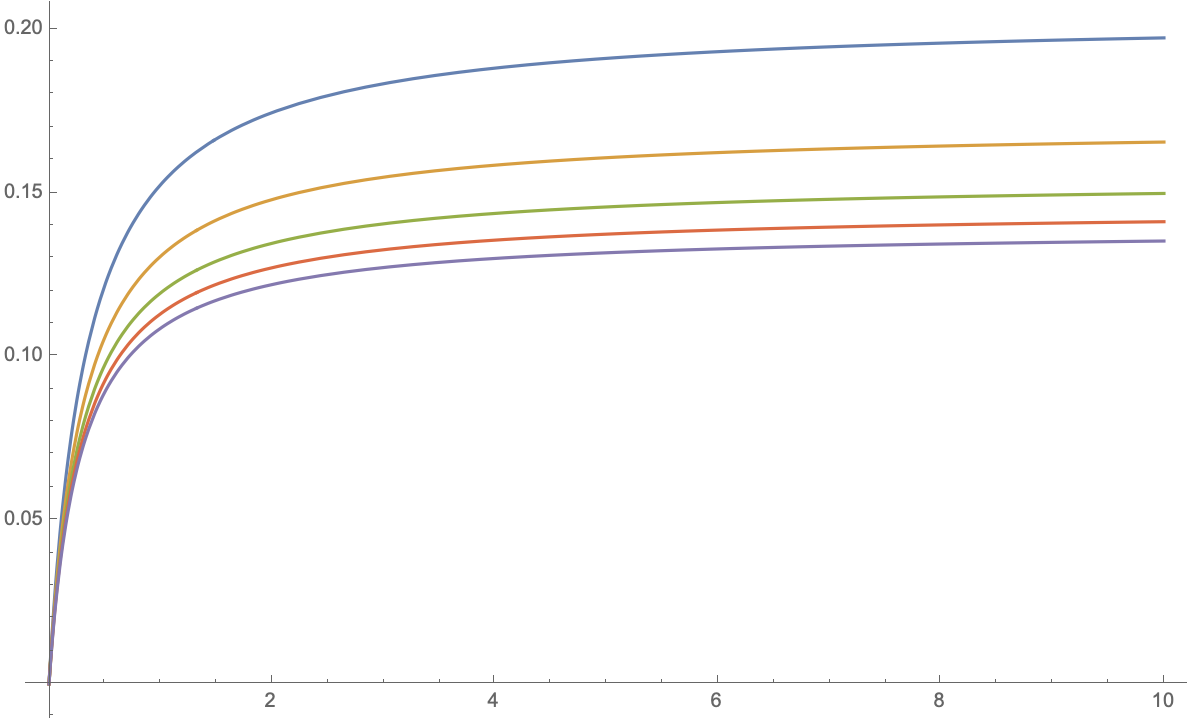}};
\node at (4.1,0) {\includegraphics[width=0.51\textwidth]{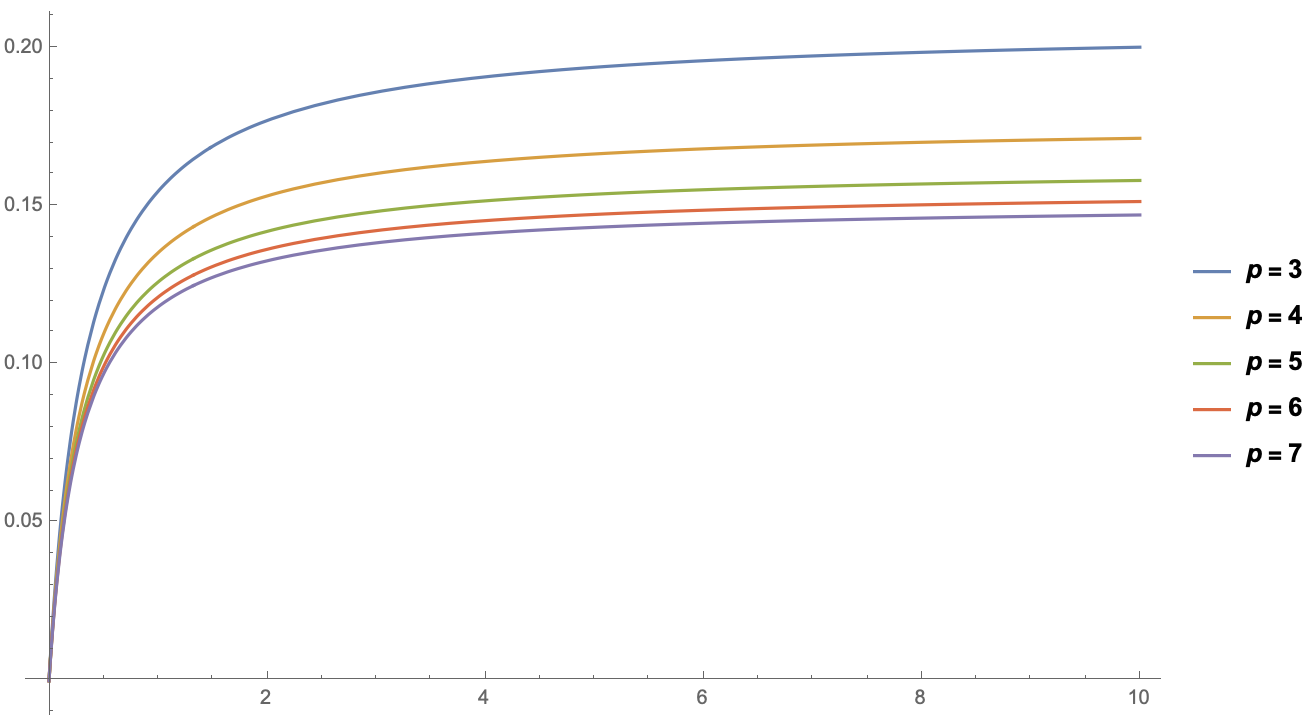}};

\node[font=\small]  at (7.1,-1.8) {$\beta$};
\node[font=\small]  at (-.5,-1.8) {$\beta$};
\node  at (.7,2.1) {$\frac{u_c(\beta)}{E_{\infty,p}}$};
\node  at (-7,2.1) {$\frac{u_c(\beta)}{E_{0,p}}$};
\end{tikzpicture}

\vspace{-.4cm}
\caption{Lower bounds on the fraction of the ground state (left) and threshold (right) energies, that we guarantee $X_t$ reaches in $O(1)$ time, uniformly over all $X_0$.}
\label{fig:quantitative-uc}
\end{figure}
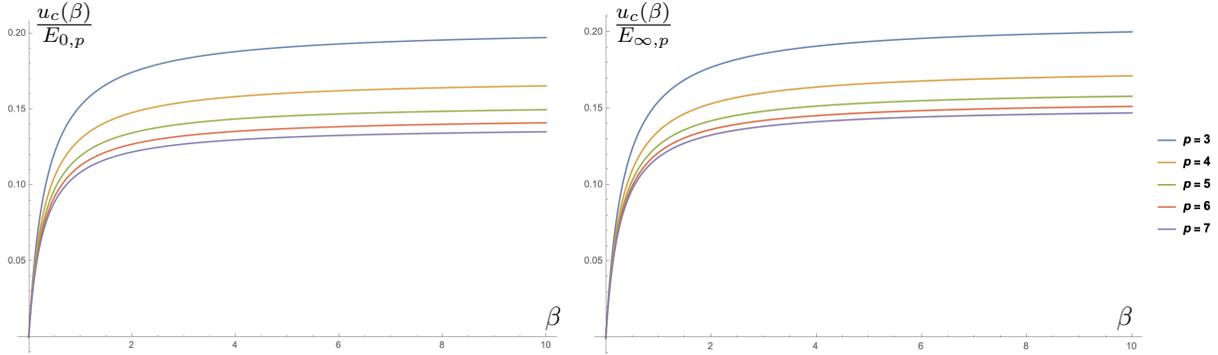

One can in fact obtain a more precise absorbing region for the dynamics which holds on the same time scale
$T_0$ as that described in \prettyref{thm:performance-guarantee}. This absorbing region is related to the fixed points of the bounding flow and the transit between these fixed
points by the flows. This is illustrated in \prettyref{fig:approx-pd} and described in detail in Section~\ref{subsec:phase-portrait-defs}. We note here, importantly, that the bounding ``ellipsoid" from  \prettyref{fig:approx-pd}, 
shrinks to a point as $\beta\to0$, whereas when $\beta\to\infty$, the upper curve converges
to the vertical line, $u=u_c$, defined by the event in \prettyref{thm:performance-guarantee} (that is, the fixed point for the 
``upper bounding flow" diverges at low temperature). 

Let us now turn to the behavior of the dynamics near critical points.
As a second consequence of the bounding flows approach, we prove that critical points of negative energy---including those with indexes of order 
$N$---repel $X_t$ to regions of larger gradients as well as \emph{larger} energies. The fact that the distance to the critical point is 
repulsive is explained by the fact that diffusions do not hit small balls in high dimensions; but the fact that the energy increases rapidly even 
when the Hessian  has $(\frac 12-\eps)N$ negative eigenvalues---and thus directions of lower energy---is surprising.

To be precise, for $(u_0,v_0)\in\R^2_+$,  let 
\begin{equation}\label{eq:omega-def}
\Omega_N(u_0,v_0) = \{x\in\cS^N: H(x) < -u_0 N, \, \abs{\nabla H(x)}^2<v_0 N\}\,.
\end{equation}
We obtain the following answer to the third question.
\begin{theorem}[Climbing saddles and wells] \label{thm:shoot-from-critical-points}
Fix any $\beta>0, p>2$. For every $\eta>0$ and every $\delta<p\eta \beta^{-1}$, there exists $c_1,c_2,\rho>0$,
such that $\mathbb P$-almost surely,
\[
\lim_{N\to\infty}\,\inf_{x\in \Omega_N(\eta,\delta) } Q_{x}\left(\left\{
\begin{array}{l}H(X_t)-H(X_0)>c_1 t N \\ \abs{\nabla H(X_t)}^2-\abs{\nabla H(X_0)}^2 > c_2 t N \end{array} \text{ for all } t\in[0,\rho]\right\}\right) = 1\,.
\]

\end{theorem}

In fact, it is clear from the above that, not only does the dynamics leave critical regions quickly,
it cannot even get close to them. For more on this, see
Section~\ref{sec:phase-portrait}. 

As an answer to the fourth question, we can obtain  trajectory-wise and pointwise comparisons as consequences of the differential inequalities allowing us to confine $U(t)$ started from $U(0)$ by the bounding flows started from the same initial data (see \prettyref{sec:comparison-theory}).

We end this section with the following important remarks.

\begin{remark}\label{rem:e-infty}
Recall that $-E_{\infty,p} N$ is the threshold energy
as in \cite{ABC13}, i.e., the energy level below which the expected number of local minima
diverges exponentially in $N$. It is believed in the physics literature (see e.g.,~\cite{Bi99}) that in order $1$ times, starting from a uniform point on $\cS^N$, the complexity threshold $-E_{\infty,p} N$ is where $H(X_t)$ gets stuck and where short-time aging is exhibited, as the critical points there have Hessians with very few negative eigenvalues, each of small magnitude. 

\prettyref{thm:performance-guarantee} and
\prettyref{thm:shoot-from-critical-points}, do not precisely identify where (or if) $H(X_t)$  gets stuck on order~$1$ time scales. However, they indicate that if the slowdown and aging at $-E_{\infty,p} N$ are due to the few flat directions that ``point down" near critical points, $X_t$ needs to maintain a large modulus of its gradient throughout its descent, and needs to stay macroscopically far from the critical points themselves.
\end{remark}
\begin{remark}\label{rem:ansatz}
We note here the following interesting observation about the differential system in Theorem~\ref{thm:bounding-flows-thm}. Under the ansatz that $G(\nabla H,\nabla H)(X_t) =O(\sqrt{N})$ for all order one times, $\Lambda_p$ disappears, the differential inequalities become equalities, and $(u_N,v_N)$ becomes fully autonomous. In this case, the resulting fixed point is at $(\beta,p)$. Surprisingly, this corresponds to $(\langle \frac{-H}{N}\rangle,\langle \frac{|\nabla H|^2}{N}\rangle)$ in the entire replica symmetric regime ($\beta<\beta_c$), where $\langle \cdot \rangle$ denotes expectation under $\pi$. In particular, under this ansatz, this would occur even in a window $\beta \in (\beta_d, \beta_c)$ where there is slow mixing.

Given that $\nabla H(x)$ and $G(x)$
can be shown to be independent, one can show that if $X_0$ is chosen
uniformly over $\cS^N$, the ansatz holds for $t=0$; moreover the ansatz is true under the Gibbs measure $\pi$ whenever $\beta<\beta_c$. One might thus hope to show that when $\beta$ is small, from a start chosen uniformly at random, $G(\nabla H,\nabla H)(X_t)$ remains microscopic throughout the trajectory of $X_t$ until equilibration, and the system $(u_N(t),v_N(t))$ closes. 
We note that this cannot be the case when $\beta>\beta_c$, where the Gibbs expectation of $G(\nabla H,\nabla H)(x)$ must in fact be order $N$ to guarantee that $(\langle \frac{-H}{N}\rangle,\langle \frac{|\nabla H|^2}{N}\rangle)$ is still a fixpoint of the system when $\langle\frac{-H}{N}\rangle \neq \beta$.

If desired, this could be extended further by starting from a point chosen uniformly at random amongst points of a given energy and modulus of gradient. Again, the independence of $G$ from both $H$ and $\nabla H$ would imply that the ansatz above holds with high probability at $t=0$. For such initializations, including the uniform at random initialization, one could then improve the bounding flows and narrow the corridors to which the dynamics are confined in short times, by using the bound $|G(\nabla H,\nabla H)(X_t)|\leq \Lambda_p' t v$ instead of the sup-norm bound of $\Lambda_p v$ used in~\eqref{eq:cf-def}.

\end{remark}

\subsection{Previous results}\label{sec:history}
The dynamics of mean field spin glass models are expected
to exhibit a deep and rich structure on both short and long timescales. 
At the level of convergence to equilibrium, when $\beta$ is small, the measure $\pi$ admits a logarithmic Sobolev inequality with constant $c>0$ which
is order $1$ in $N$ \cite{GJ16}, so the dynamics converges to $\pi$
in order $1$ time. For large $\beta$, if $L$ is the infinitesimal generator of $X_t$, the spectral
gap of $-L$ decays exponentially in $N$ \cite{GJ16, BAJag17}
so that the process $X_{t}$ reaches equilibrium only after a time that is exponential in $N$. 
For a physics approach to the analysis of the spectrum of $-L$, see  \cite{biroli2001metastable}. 
For studies of the spectrum and spectral gap
in related spin glass models see,  \cite{Mathieu,BovFag05,bauerschmidt2017very}.

At low temperatures, there has been extensive research on the behavior of glassy dynamics on timescales that
are exponentially large but shorter than the time to equilibrium, sometimes referred to as \emph{activated dynamics}. 
 On such timescales, glassy dynamics have traditionally been found to exhibit the important feature of \emph{aging}. 
Following the influential works of \cite{BouDean95,Bou92} in the physics literature, this question was studied
for Random Hopping time dynamics of the Random Energy Model (REM) in \cite{BABG02,BABoGa2,BABoGa},
and for the $p$-spin model in  \cite{BAGun12,BABoCe,BoGa}. 
More recently, aging has been established for Metropolis dynamics of the REM
 in the recent important works \cite{CernyWassmer,Gayr16}. For an analysis of Glauber dynamics
 of the REM in the physics literature, see \cite{baity2018activated}.

These dynamics are also interesting in timescales much shorter than the equilibration time, where the dynamics
is expected to exhibit the phenomenon of short-time aging at low temperatures.
This picture was first suggested 
in the physics literature in \cite{somp82}. For spherical $p$-spin models, 
the classical way to study dynamics in this regime is to introduce the two-time autocorrelation 
\begin{align*}
C_N(t,s) = \frac 1N \sum_i X_t^i X_s^i\,,
\end{align*}
corresponding to the overlap between $X_t$ and $X_s$. 
It is possible to find a second two--time function, the response function $R_N(t,s)$, such that the pair $(C_N,R_N)$ converges, as $N\to\infty$, to the solution of 
a system of coupled integro--differential equations. This was  proposed 
 in~\cite{crisanti1993sphericalp,CugKur93} and is now commonly referred to as the Cugliandolo--Kurchan equations. The main focus in the 
 physics literature has been to understand what these equations imply about $C_N(t,s)$ when $t$ and $s$ tend to infinity together, 
 after the thermodynamical $N\to\infty$ limit has been taken; in the low temperature regime, aging can be seen in the behavior of 
 $C_N$ in this limit.
This picture was established rigorously in the simpler
case of $p=2$ in \cite{BADG01} where the equations for $C$ and $R$ decouple. 
These equations have also been proven to hold
for $p\geq 3$ in \cite{BADG06}; however, here the decoupling is not expected and
it remains a deep and challenging question to prove aging.

One can also study the evolution of the energy, $H(X_t)$, via the Cugliandolo-Kurchan equations. However, 
due to the difficulty of the equations for $(C_N,R_N)$ when $p>2$, this expression is not particularly amenable to rigorous analysis 
without an ansatz on the form of $C$ and $R$. Moreover, the study of these equations is restricted to certain random initial 
conditions, typically the cases where the initial distribution of $X_0$ is well--concentrated and independent of the field $H$.
Observe that the approach through the Cugliandolo--Kurchan equations is exact, whereas the 
bounding flows approach pays the price of inequalities. In exchange, this allows us to analyze \emph{one-time} observables, from which we can glean meaningful information about the behavior of $X_t$. Furthermore, this information can depend in a precise way on the initial data, which we are now free to choose as we wish.
Finally, we observe that the analysis of the Hessian as it effects the dynamics has also been developed using the Cugliandolo--Kurchan approach in the physics literature in \cite{kurchan1996phase}.

\section{Outline and ideas of proofs}\label{sec:outline-and-ideas}
In this section, we sketch out a derivation of the differential inequality of Theorem~\ref{thm:bounding-flows-thm}, then formalize the phase diagram depicted in Figure~\ref{fig:approx-pd} in Theorem~\ref{thm:full-phase-portrait}, from which we will directly conclude Theorems~\ref{thm:performance-guarantee} and~\ref{thm:shoot-from-critical-points}.  
The rest of the paper will then develop the necessary regularity theory and make the derivation of the differential inequality rigorous, before turning to comparison theory of bounded differential systems and proving Theorem~\ref{thm:full-phase-portrait}. 

\subsection{Derivation of Eq.~\eqref{eq:diff-ineq-main}}\label{subsec:derivation-diff-ineq}
We will sketch the derivation of the differential inequality of Theorem~\ref{thm:bounding-flows-thm}. Making this rigorous will be the project of Sections~\ref{sec:quasiautonomy}--\ref{sec:diff-ineq}.
 Let us begin first by considering the evolution of $H(X_t)$. Using that  $p$-spin Hamiltonian is a homogeneous function of order $p$, one can see that is an approximate spherical harmonic, $\Delta H = - p H +o(N)$, where the error term comes from the trace of the Euclidean Hessian (see  \eqref{eq:g-def} and \eqref{eq:tr-goe}). Consequently, by Ito's lemma, 
\begin{align*}
u_N(t)& \approx u_N(0)+\int_0^t \Big(p\frac{H(X_s)}{N} + \beta \frac{|\nabla H(X_s)|^2}{N}\Big)ds+M^u_t\,. 
\end{align*}
where $M_t^u$ is a martingale term. 
Given the above integral equation, it is then natural to consider the evolution of $|\nabla H(X_t)|^2$.  Again using Ito's lemma, 
\begin{align*}
v_N(t) = v_N(0) + \frac{1}{N}\int_0^t \Big(\Delta\abs{\nabla H}^2(X_s) - 2 \beta \nabla^2 H(X_s )(\nabla H(X_s),\nabla H(X_s))\Big) ds + M_t^{v}\,,
\end{align*}
where $M_t^v$ is again a martingale term. 
Examining the drift terms above, we will find that 
\begin{align*}
\Delta |\nabla H(X_s)|^2 \approx 2p(p-1) N + 2p^2 \frac 1{N} (H(X_s))^2 - 2(p-1) |\nabla H(X_s)|^2 \,,
\end{align*}
by elementary geometric considerations on the sphere and Bochner's formula (see~\eqref{eq:bochner}). 
At the same time, if we let $G$ be the Euclidean Hessian restricted to $T\cS^N$ as in Remark~\ref{rem:lambda-u-c}, then
\begin{align*}
\nabla^2H(X_s)(  \nabla H(X_s), \nabla H(X_s)) \approx G(\nabla H,\nabla H)(X_s) - p\frac{1}{N} H(X_s)|\nabla H(X_s)|^2 \,.
\end{align*} 
At this point, bounding the operator norm of $G$ by $\Lambda_p$ per~\eqref{eq:Lambda-p-defn}, we then find that 
\begin{align*}
v_N(0) + \int_0^t \mathcal F_2^L(u_N(s),v_N(s))ds  -o(1) \leq v_N(t)\leq v_N(0)+ \int_0^t \mathcal F_2^U(u_N(s),v_N(s))ds +o(1)\,.
\end{align*}
The Sobolev estimates for $H$ developed in Section~\ref{sec:regularity} will establish that the approximations above hold $\mathbb P$-eventually almost surely, uniformly over all choices of initial data, and will imply that both martingale terms are of order $O(\sqrt{t/N})$. Sending $N\to\infty$, we deduce~\eqref{eq:diff-ineq-main}. 

\subsection{Phase portrait of Figure~\ref{fig:approx-pd} via bounding flows}\label{subsec:phase-portrait-defs}
We now describe the precise phase diagram we can deduce from the differential inequalities above for all the subsequential limits of $(-\frac{H(X_t)}{N},\frac{|\nabla H(X_t)|^2}{N})$. Let us begin by decomposing the phase space into a set of regions in terms of which our phase diagram depicted in Figure~\ref{fig:approx-pd} is defined. (As $v_N$ is deterministically nonnegative, we are always restricting our flows to $\mathbb R \times \mathbb R_+$.) These regions will be described by certain zero-sets of the lower and upper bounding flows governed by the differential inequality in Theorem~\ref{thm:bounding-flows-thm}. To that end, recall the definitions of $\mathcal F_1$ and $\mathcal F_{2}^{L/U}$ from~\eqref{eq:cf-def} for any  $(x_{0},y_{0})$ in  $\R^2$,
let $A_{L}(t)$ and $A_U(t)$ denote the solutions to 
the initial value problems
\begin{equation}\label{eq:A-L-def}
\begin{cases}
\dot{A_L}(t)=\big(\mathcal F_1(A_L(t)), \mathcal F_2^L(A_L(t))\big)\\
A_L(0)=(x_{0},y_{0})
\end{cases}, \qquad \mbox{and}\qquad 
\begin{cases}
\dot{A_U}(t)=\big(\mathcal F_1(A_U(t)), \mathcal F_2^U(A_U(t))\big)\\
A_L(0)=(x_{0},y_{0})
\end{cases}.
\end{equation}
(As both $\Phi_L$ and $\Phi_U$ are locally 
Lipschitz, existence and uniqueness are ensured.)
Call $A_L(t)$ the \emph{lower bounding flow} and $A_U(t)$ 
the \emph{upper bounding flow}.
Now, define (refer to Figure~\ref{fig:curves})
\begin{equation}
\begin{aligned}
f_L(u) & =\frac{p(p-1)+p^{2}u^{2}}{p-1+p\beta u+\beta\Lambda_{p}}\,,\\
f_{U}(u)& =  \frac{p(p-1)+p^{2}u^{2}}{p-1+p\beta u-\beta\Lambda_{p}}\,,\\
\ell_1 (u)& = pu\beta^{-1}\,.
\end{aligned}
\end{equation}
where we impose domains $\{u:0\leq f_L(u)<\infty\}$ and $\{u:0\leq f_U(u)<\infty\}$ on the first two. 
We observe the following important facts about these functions. 

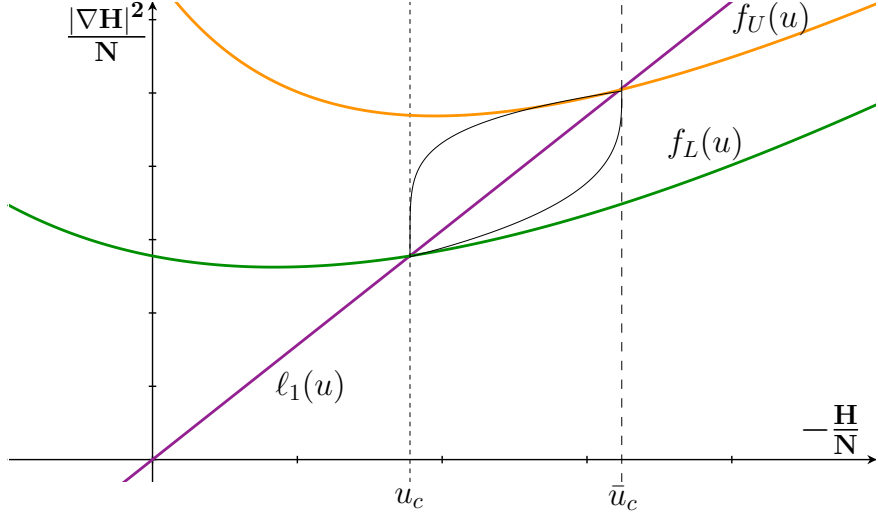
\begin{figure}
\centering
\resizebox{12cm}{!}{
\begin{tikzpicture}
\node at (0,0) {\includegraphics[width=0.8\textwidth]{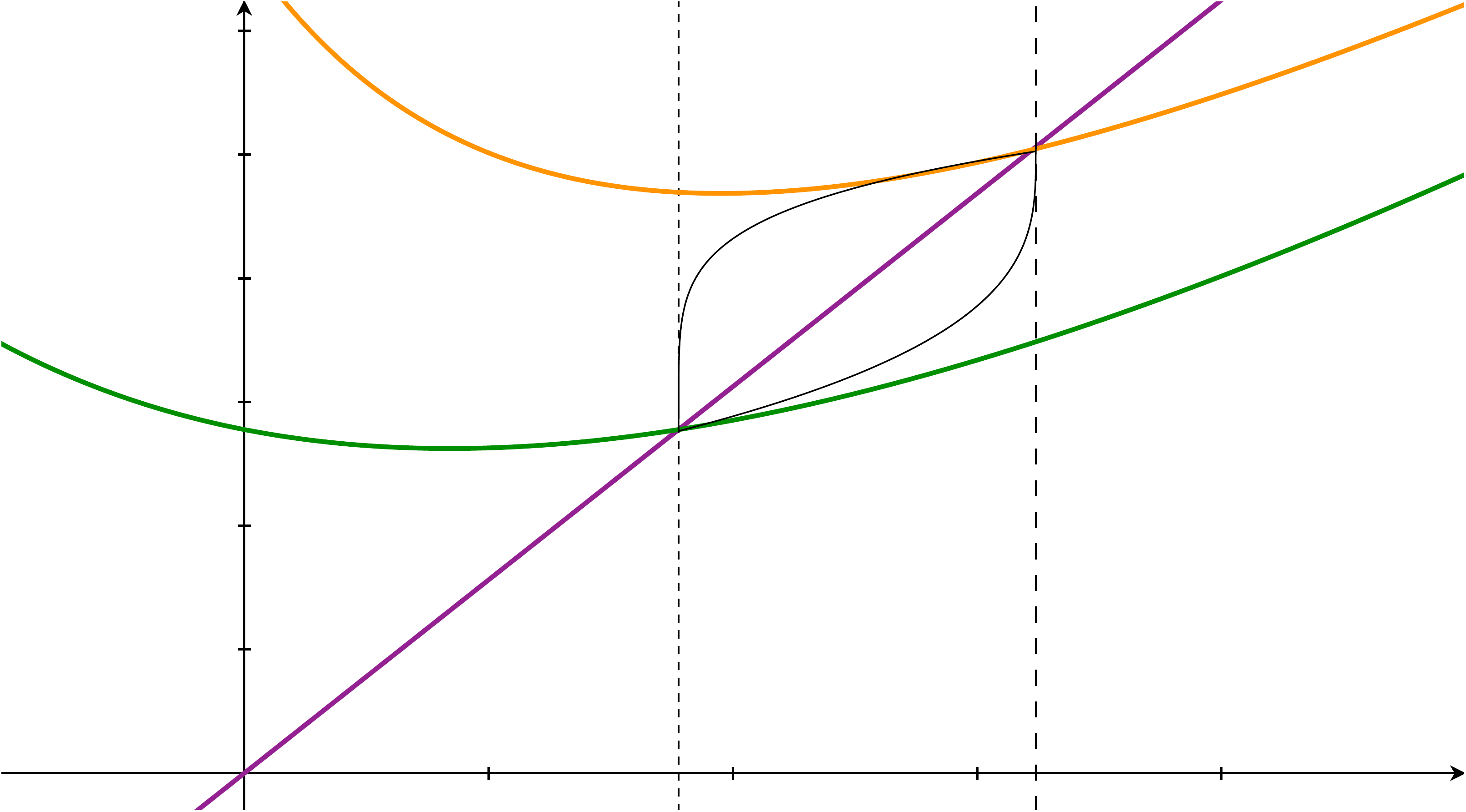}};
\node [font=\Large] at (-.5,-3.9) {$u_c$};
\node [font=\Large] at (2.75, -3.9) {$\bar u_c$};
\node [font=\Large] at (5,3.4) {$f_U(u)$};
\node [font=\Large] at (-2,-2.2) {$\ell_1(u)$};
\node [font=\Large] at (4,1.5) {$f_L(u)$};
\node [font=\LARGE] at (5.9,-2.85) {$\mathbf{-\frac HN}$};
\node [font=\LARGE] at (-5.1,3.2) {$\mathbf{\frac{|\nabla H|^2}N}$};
\end{tikzpicture}
}
\vspace{-.4cm}
\caption{The curves $\ell_1(u)$ (purple), $f_L(u)$ (green), and $f_U(u)$ (orange). The two black curves are $\gamma_U(u;z_c)$ and $\gamma_L(u;\bar z_c)$ above and below $\ell_1(u)$ respectively.}\label{fig:curves}
\end{figure}

\begin{remark}\label{rem:fixpoints}
For each $u$ in the domain of $f_L(u)$, the map $v\mapsto \cF_2^L(u,v)$ 
has a unique zero. In particular,  $f_L(u)$ is the unique function such that $\cF_2^L(u,f_L(u)) = 0$. 
Likewise for $f_U(u)$ with respect to $\cF_2^U$. 
It follows from this that the unique (attractive) fixpoint of $A_L$ is given when $f_L(u)=\ell_1 (u)$ and by explicit calculation, we see this happens at $z_c =(u_c,f_L(u_c))$ for $u_c$ defined in~\prettyref{eq:u_c}. Similarly, for $\bar u_c$ defined as 
\begin{equation}\label{eq:bar-u_c}
\bar{u}_{c}=\begin{cases}
\beta\left(1-\tfrac{\beta\Lambda_{p}}{p-1}\right)^{-1} & \mbox{if }\beta\Lambda_{p}<p-1\\
\infty & \mbox{if }\beta\Lambda_{p}\geq p-1
\end{cases}\,,
\end{equation}
when $\bar u_c <\infty$, $\bar z_c=(\bar u_c,f_U(\bar u_c))$ is the unique (attractive) fixpoint of $A_U$. Finally, we remark that by explicit calculation, one can see that $f_L'(u)>0$ for all $u\geq u_c$ and likewise, $f_U'(u)>0$ for all $u\geq \bar u_c$ when $\bar u_c<\infty$. 
\end{remark}

Let us also use the following hitting time notation. 
\begin{definition}
For fixed initial data $(x_0,y_0)$ 
and a set $B\subset \mathbb R^2$, let 
\begin{align*}
\tau^L_B & = \inf\{t\geq 0 : A_L(t)\in B\}\,, \qquad \mbox{and}\qquad \tau^U_B  = \inf \{t\geq 0: A_U(t) \in B\}\,.
\end{align*}
and for a dynamical system $U:[0,T]\to \mathbb R\times \mathbb R_+$ with $U(0)=(x_0,y_0)$,  let 
\begin{align*}
\tau_B= \tau_B(U(t)) = \inf\{t\geq 0: U(t)\in B\}\,.
\end{align*}
\end{definition}
We omit the dependence of the above on 
 $U(t)$ and $(x_0,y_0)$ when unambiguous. We use the convention that whenever hitting times are infinite, corresponding events are vacuously satisfied. 
 
The bounding dynamics depend crucially on whether $(x_0,y_0)$ is above or below $\ell_1(u)$. Let 
\begin{align}
W_{+}=\left\{(x,y)\in \mathbb R\times\mathbb R_+: \cF_{1}(x,y)>0\right\}\quad \mbox{and}\quad W_{-}=\left\{ (x,y)\in \mathbb R\times\mathbb R_+: \cF_{1}(x,y)<0\right\}\,. \label{eq:W-pm}
\end{align}
Observe that for $(x_{0},y_{0})\in W_{+}$, we have $\tau_{W_{+}^{c}}^{L}>0$
and $(A_{L}(t))_{t\leq\tau_{W_{+}^{c}}^{L}}$ is the graph of a function,
call it $\gamma_{L}(u;(x_{0},y_{0}))$ with domain $\mbox{Dom}(\gamma_{L},(x_{0},y_{0}))$.
The same holds for $(x_{0},y_{0})\in W_{-}$. Define $\gamma_{U}(\cdot;(x_{0},y_{0}))$
and $\mbox{Dom}(\gamma_{L}, (x_0,y_0))$ analogously for $A_{U}(t)$.

These allow us to define the subsets of $\mathbb R\times \mathbb R_+$ with respect to which we prove a phase diagram. The absorbing set for our dynamics will approximately be $A_{0}$, which is defined as follows:
\begin{equation}
A_{0}=\begin{cases}
\left\{ (u,v):v\in[\gamma_{L}(u;A_{U}(\tau_{W_{+}^{c}}^{U};z_c)),\gamma_{U}(u;z_c)]\right\}  & \mbox{if }\bar{u}_c<\infty\\
\left\{ (u,v):v\in[f_{L}(u),\gamma_{U}(u;z_c)]\right\}  & \mbox{otherwise}
\end{cases}\,.\label{eq:A_0}
\end{equation}

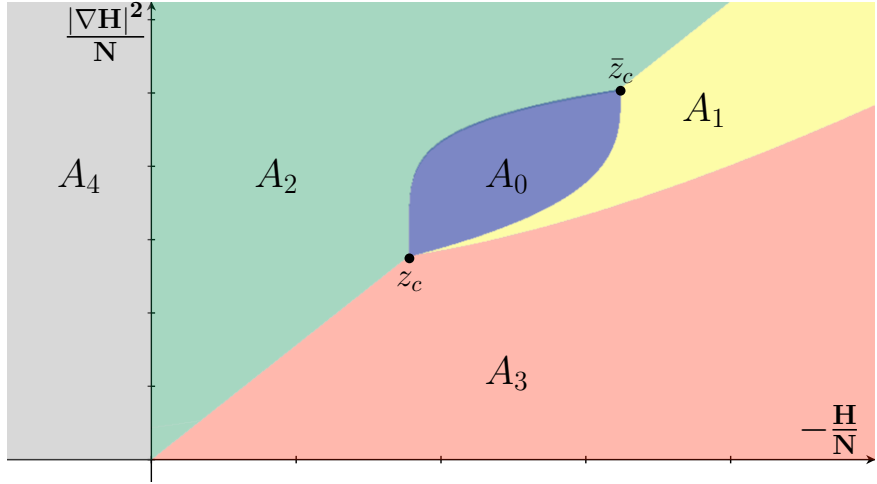
\begin{figure}
\centering
\resizebox{12cm}{!}{
\begin{tikzpicture}
\node at (0,0) {\includegraphics[width=0.8\textwidth]{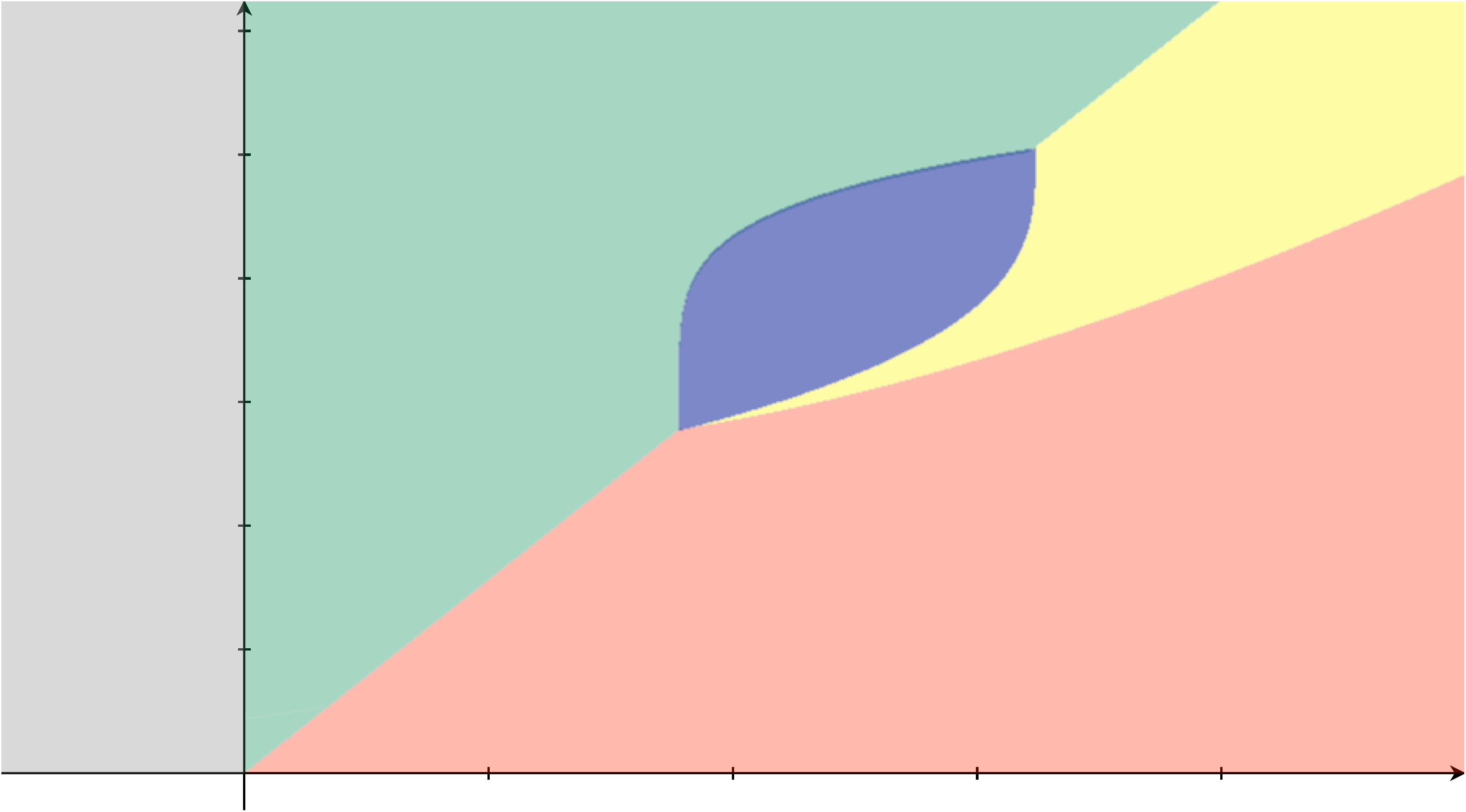}};

\node [font=\LARGE] at (-2.5,1) {$A_{2}$};
\node [font=\LARGE] at (1,1) {$A_{0}$};
\node [font=\LARGE] at (4,2) {$A_{1}$};
\node [font=\LARGE] at (1,-2) {$A_{3}$};
\node [font=\LARGE] at (-5.5,1) {$A_{4}$};
\node [font=\Large] at (-.48,-.25)[circle,fill,inner sep=1.5pt]{};
\node [font=\Large] at (-.48, -.6) {$z_c$};
\node [font=\Large] at (2.73,2.3)[circle,fill,inner sep=1.5pt]{};
\node [font=\Large] at (2.73, 2.6) {$\bar{z}_c$};
%
%
\node [font=\LARGE] at (5.9,-2.85) {$\mathbf{-\frac HN}$};
\node [font=\LARGE] at (-5.1,3.2) {$\mathbf{\frac{|\nabla H|^2}N}$};

\end{tikzpicture}
}
\caption{A diagram of the sets $A_0,...,A_4$. From any initial $(x_0,y_0)$, $\mathscr L$-almost surely, $U(t)$ finds and stays within distance $\epsilon$ of $A_{0}$ in finite time.}\label{fig:partition}
\end{figure}

Notice that the $u$--values here are implicitly constrained by the domains of 
of $\gamma_{L}$ and $\gamma_U$. As we wish to prove certain hitting times are finite, it will be helpful to enlarge
$A_{0}$ by $\delta$ near the fixpoints of $A_L$ and $A_U$. To this end, let
\begin{equation}
A_{0,\delta}=A_{0}\cup B_{\ell_{\infty}}(z_{c},\delta)\cup B_{\ell_{\infty}}(\bar{z}_{c},\delta)\,,\label{eq:A_0-eps}
\end{equation}
where if $\bar{u}_{c}=\infty$, the second ball is empty. 
Define the following other regions:
\begin{equation}\label{eq:A_i}
\begin{aligned}
A_{4}= & \{(u,v):u<0\}\,, \\
A_{3}= & \{(u,v):v\leq\ell_{1}(u)\wedge f_{L}(u)\}\,, \\
A_{2}= & \{(u,v):v>pu\beta^{-1}\}\setminus\left(A_{0}\cup A_{4}\right)\,, \\
A_{1}= & \{(u,v):v\in[f_{L}(u),\ell_{1}(u)]\}\setminus A_{0}\,.
\end{aligned}
\end{equation}
Note that if $\beta\Lambda_{p}$ is sufficiently large and $\bar u_c =\infty$,
we have $A_{1}=\emptyset$. We invite the reader to refer to and use Figure~\ref{fig:partition} as a guide throughout this section and Section~\ref{sec:phase-portrait}.

\begin{remark}\label{rem:A-0-definition}Since $\mathcal{F}_{2}^{U}(u,\ell_{1}(u))>0$ for
all $u<\bar{u}_{c}$ while $\mathcal{F}_{2}^{L}(u,\ell_{1}(u))<0$
for all $u>u_{c}$ the set $A_{0}$ is well-defined and contains the
entire segment $\{(u,v):u\in[u_{c,}\bar{u}_{c}],v=\ell_{1}(u)\}$. (In particular this implies that the union $ \bigcup_{i=0}^4 A_i =\mathbb R \times \mathbb R_+$.)
If we also consider the fact that $\mathcal{F}_{2}^{L}(u,v)\geq0$
for all $v\leq f_{L}(u)$ we deduce that $A_{0}\subset\{(u,v):u\geq u_{c},v\geq pu_{c}\beta^{-1}\}$. 
\end{remark}

We are now in position to describe the phase portrait for $U_N(t)$. It will be useful to restrict our attention to bounded subsets of $\mathbb R \times \mathbb R_+$.  The phase portrait will describe how dynamical systems satisfying the differential inequality of Theorem~\ref{thm:bounding-flows-thm}  and a boundedness assumption move between the sets $(A_i)_{i=0,...,4}$ in finite time. For two vectors $v,w\in\R^{2}$ we say $v\leq_{2}w$ if $v_{1}=w_{1}$
and $v_{2}\leq w_{2}$.

\begin{definition}\label{def:condi}
Let $U:[0,T]\to\R^{2}$ be $C^{1}$. We say that $U$ satisfies \textbf{Condition
I} if 
\begin{equation}
\Phi_{L}(U(t))\leq_{2}\dot{U}(t)\leq_{2}\Phi_{U}(U(t))\,.\label{eq:Diff-ineq}
\end{equation}
\end{definition}

\begin{definition}\label{def:condb}
We say that $U:[0,T]\to\R^{2}$ satisfies \textbf{\condB} if there
exists a compact (nonempty) $W\subset\mathbb{R}\times\mathbb{R}_{+}$ such that if
$U(0)\in W$ then $U(t)\in W$ for all $t\leq T$. 
\end{definition}

For ease of notation, we let $A_{i}$ denote the
sets from~\eqref{eq:A_0-eps}--\eqref{eq:A_i} restricted to the compact set $W$ for which
$U$ satisfies \condB. Recall that $\tau_{A}$ is
the hitting time of $A$ for $U$ started from $U(0)=(x_0,y_0)$. The phase diagram depicted in Figure~\ref{fig:approx-pd} can be described formally as follows.


\begin{theorem}\label{thm:full-phase-portrait}
Let $U:[0,T]\to \mathbb R^2$ satisfy \condI~ and \condB. For every $\epsilon>0$, there exists $\delta>0$ and $T_0$ such that the following holds: 
\begin{enumerate}
\item When $U$ starts in $A_{4}$, it exits into $A_2$ in time at most $T_0$. \label{item:A4}
\item When $U$ starts in $A_3\setminus A_{0,\delta}$, it exits into $A_{0,\delta}\cup A_1 \cup A_2$ in time at most $T_0$. \label{item:A3}
\item When $U$ starts in $A_2\setminus A_{0,\delta}$, it exits into $A_{0,\delta}\cup A_1$ in time at most $T_0$. \label{item:A2}
\item When $U$ starts in $A_{1}\setminus A_{0,\delta}$, it exits into $A_{0,\delta}$ in time at most $T_0$.  \label{item:A1}
\item There exists a set $\mathcal A_\epsilon$ having $A_{0,\delta}\subset \mathcal A_\epsilon \subset [u_c-\epsilon,\infty)\times [v_c-\epsilon,\infty)$ such that when $U$ starts in $\mathcal A_\epsilon$, it stays in $\mathcal A_\epsilon$ for all times.\label{item:absorbing}
\end{enumerate}
\end{theorem}

\begin{remark}The set $\mathcal A_\epsilon$ will be a non-uniform enlargement of the set $A_0$ by some $O(\epsilon)$ distance; in particular, the enlargement is different around the corners of the set $A_0$, and is defined precisely in~\eqref{eq:mathcal-A-eps} of Section~\ref{sec:phase-portrait}, while proving Theorem~\ref{thm:full-phase-portrait}. 
\end{remark}

We defer the proof of Theorem~\ref{thm:full-phase-portrait} to Section~\ref{sec:phase-portrait} and its proof requires trajectory-wise and point-wise comparisons that are developed in Section~\ref{sec:comparison-theory} and may be independently interesting. With Theorem~\ref{thm:full-phase-portrait}, however, we are in position to easily complete the proofs of Theorems~\ref{thm:performance-guarantee} and~\ref{thm:shoot-from-critical-points}.

%
%
%
%
%

\subsection{Proofs of Theorems~\ref{thm:performance-guarantee}--\ref{thm:shoot-from-critical-points}}
We now read off our main results from Theorem~\ref{thm:full-phase-portrait}, up to a regularity estimate ensuring that \condB~is satisfied. Since we study properties of all possible limit points
for a sequence of initial data, we introduce the following notation
for convenience.


\begin{definition}\label{def:s-U}
For every sequence $x_N \in \cS^N$, let $\sU((x_N))$ denote the set of all possible limiting laws of $(u_N(t),v_N(t))_{t\in [0,T]}$ started from $X(0)=x_N$. Notice the implicit dependence on $T$ here.
\end{definition}

\begin{proof}[\textbf{\emph{Proof of Theorem~\ref{thm:performance-guarantee}}}]
First observe by Theorem~\ref{thm:bounding-flows-thm} that $\mathbb P$-eventually almost surely, for every $T$, every sequence $(x_N)\in \cS^N$, and every limit law $\sL\in \sU((x_N))$, $\sL$-almost surely, $U\sim \sL$ satisfies \condI. It will follow from standard $L^\infty$ bounds on ${H}/{N}$ and ${|\nabla H|^2}/N$ (special cases of Theorem~\ref{thm:reg}) that there exists $W\subset \mathbb R \times \mathbb R_+$ such that $\mathbb P$-eventually almost surely, for every $T$, every sequence $(x_N)\in \cS^N$, and every limit law $\sL\in \sU((x_N))$, $\sL$--almost surely, $U(0)\in W$ and \condB~holds.  

Now fix any $\epsilon>0$ and let $\delta>0$ be that given by \prettyref{thm:full-phase-portrait}. Let $(x_N)$ denote a sequence of near minimizers
of \eqref{eq:performance-guarantee} and let $\mathscr L\in \sU((x_N))$. As noted above, $\sL$-a.s., $U\sim\mathscr L$ satisfies the conditions of \prettyref{thm:full-phase-portrait}. Consequently, by stitching together items~\eqref{item:A4}--\eqref{item:A1} in Theorem \ref{thm:full-phase-portrait},
there is a $T_{0}$ such that $\tau_{A_{0,\delta}}\leq 4T_{0}$. By item~\eqref{item:absorbing} of 
\prettyref{thm:full-phase-portrait}, $\tau_{\cA_{\epsilon}}\leq\tau_{A_{0,\delta}}$
and thus, for every $T>4T_0$, 
\[
\sL\left(\Big\{ \inf_{t\in[4T_0,T]}u(t)>u_{c}-\epsilon\Big\} \cap\Big\{ \inf_{t\in[4T_0,T]}v(t)\geq v_{c}-\epsilon\Big\} \right)=1\,.
\]
As this event is strongly open in $C([0,T])^2$, it
follows by the Portmanteau Lemma that 
\[
\lim_{N\to\infty}Q_{x_{N}}\left(\Big\{ \inf_{t\in[4T_0,T]}u(t)>u_{c}-\epsilon\Big\} \cap\Big\{ \inf_{t\in[4T_0,T]}v(t)\geq v_{c}-\epsilon\Big\} \right)=1\,.
\]
The result follows as our $(x_N)$ was such that the probability in~\eqref{eq:performance-guarantee} got arbitrarily close to its infimum over all $x\in \mathcal S^N$ as $N$ went to $\infty$.
\end{proof}
\begin{proof}[\textbf{\emph{Proof of \prettyref{thm:shoot-from-critical-points}}}]
 For each $\eta>0,$ let $\delta_{0}(\eta)=p\eta/\beta=\ell_{1}(\eta)$.
Fix any $\delta<\delta_{0}$ and let $\Omega=\left\{ (u,v):u>\delta,v<\delta\right\}$.
There exists a $c>0$ such that $\cF_{1}$ and $\cF_{2}^{L}$ from
\eqref{eq:cf-def} satisfy $\cF_{1}<-c$ and $\cF_{2}>c$ on $\Omega$. Fix any sequence $(x_N)\in \cS^N$ such that $(-H(x_N)/N,|\nabla H(x_N)|^2/N)\in\Omega$. Then for every $\sL\in \sU((x_N))$, by continuity of $\cF_{1},\cF_2^L$,
there exist $c_{1},\rho>0$ such that $\sL$-a.s.,
\[
U^{1}(t)-U^{1}(0)<c_{1}t\quad\text{ and }\quad U^{2}(t)-U^{2}(0)>c_{1}t
\]
for every $t\leq\rho$. The result then follows by the Portmanteau
Lemma. 
\end{proof}

\subsection*{Acknowledgements} The authors thank the anonymous referees for their helpful comments and suggestions. The authors thank Giulio Biroli and Chiara Cammarota for interesting discussions. This research was conducted while G.B.A.\ was supported by NSF DMS1209165 and BSF 2014019, and A.J.\ was supported by NSF OISE-1604232. R.G.\ thanks NYU Shanghai for its hospitality during the time some of this work was completed.

\section{Quasi-autonomy for Langevin-type systems on spheres in large dimensions}\label{sec:quasiautonomy}

We seek to study the properties of scaling limits of 
observables of Langevin dynamics. The key observation
is that certain observables are \emph{quasi-autonomous}.
Informally, this will mean that they asymptotically (in $N$) satisfy 
autonomous differential inequalities. 

In this section, we
introduce the notion of a quasi-autonomous family
of observables. We then present an elementary
result proving quasi-autonomy for a sufficiently regular family 
of observables evolving with respect to Langevin dynamics on spheres in large 
dimensions. 

\textbf{Notation.}  Throughout, we say that $f_N \lesssim_a g_N$ if there is a $C(a)>0$ such that $f_N \leq C(a)g_N$ for every $N$. We write $f_N= O(g_N)$ if $f_N \lesssim g_N$ and $f_N = o(g_N)$ if $f_N/g_N \to 0$ as $N\to\infty$.  

Let $(V_N)$ be a sequence
of smooth functions with $V_N \in C^\infty(\cS^N)$. 
Let $X_{t}$  be the stochastic process on $\cS^N$
with generator 
\begin{align}\label{eq:langevin-generator-general}
L=\Delta-\beta\left\langle \nabla V_N,\nabla\cdot\right\rangle.
\end{align}
Here and in the following $\langle\cdot,\cdot\rangle$ will refer
to the inner product on $T\cS^N$ given by the metric $g$ and $\Delta, \nabla$ will be the covariant Laplacian and derivative respectively. Finally, $d(x,y)$ will be the distance on $\cS^N$. 
Notice that $X_{t}$ can be seen as the solution to 
the Langevin equation with Hamiltonian $V_N$:
\begin{equation}\label{eq:langevin}
\begin{cases}
dX_t &= \sqrt{2}dB_{t}-\beta \nabla V_N(X_t) dt\\
X_{0} &= x
\end{cases}\,,
\end{equation}
where $B_{t}$
is Brownian motion on $\cS^{N}$. {Since $V_N$ is smooth, one can solve this SDE in the strong sense (see e.g., \cite{Hsu02}). } We call $X_t$ the \emph{Langevin dynamics} corresponding to $V_N$. 
Let $Q_x$ denote the law of $X_{t}$ started from $X_0=x$ and let $\bE_x$ denote the corresponding
expectation.

As $B_t$ moves on order $\sqrt{N}$ distances
in order 1 time, it is natural to consider Hamiltonians whose gradients are on this scale. To this end, we say that a sequence of functions $(V_{N}(x))$ 
with $V_N \in C^{\infty}(\cS^{N})$ is \emph{C-regular} for $C>0$
if they satisfy the following gradient estimate: if, for every~$N$, 
\begin{equation}\tag{\textbf{GE}}
\norm{\,\abs{\nabla V_{N}}\,}_{L^{\infty}(\cS^{N})}\leq C\sqrt{N}\,.\label{eq:Grad-est-U}
\end{equation}
Here and in the following, for a vector $X\in T_x \cS^N$, 
$\abs{X}$ will denote the usual norm with respect to the induced metric. 
The subscript $N$ will henceforth be omitted and we will say that $V$ is \emph{$C$-regular}.

The following lemma shows that if the sequence of Hamiltonians, $(V_N)$, is C-regular
then the ``gradient descent'' and ``diffusive'' natures of $X_t$ are on the same scale.
\begin{lemma}
\label{lem:brownian-continuity} There is a universal constant $K>0$
such that the following holds.
\begin{enumerate}
\item For all $N\geq1$, every $x\in\cS^{N}$, and every $t\geq0$, 
\[
\bE_{x}\left[d(B_{t},x)^{4}\right]\leq KN^2 t^{2}.
\]
\item If $(V_N)$ is $C$-regular, then for every $N\geq1$,
$x\in\cS^{N}$ and every $T\geq 0$ and $t\in [0,T]$, 
\[
\bE_{x}\left[d(X_{t},x)^{4}\right]\leq\left(K+ \beta^4 C^{4}T^2\right)N^2 t^{2}\,.
\]
\end{enumerate}
\end{lemma}
\begin{proof}
Recall Stroock's representation \cite[Exam. 3.3.2]{Hsu02} for $B_{t}$ which, in It\^o form, is
given by
\[
B_{t}-B_{0}=\int_0^t \Big(Id-\frac{B_{s}\tensor B_{s}}{N}\Big)dW_{s}-\int_0^t\frac{B_{s}}{N}ds\,,
\]
where $W_{t}$ is Brownian motion on $\R^{N}$, $Id$
is the identity on $\R^{N}$. Since $B_{t}\in\cS^{N}$, we have that
\begin{align*}
\abs{\frac{B_{s}}{N}}  =\frac{1}{\sqrt{N}} \qquad\text{ and }\qquad
\abs{Id-\frac{B_{s}\tensor B_{s}}{N}}_{F}^{2}  = N-1\,,
\end{align*}
where $\abs{\cdot}_{F}$ is the Frobenius norm. 
Observe that by the Burkholder--Davis--Gundy inequality
and It\^o's isometry,
\[
\bE_x\left[\abs{\int_0^t 
\Big(Id - \frac{B_s\tensor B_s}{N} \Big)dW_s}^4\right]\lesssim (N-1)^2t^2\,.
\]
Recall further
that the metric on the sphere and Euclidean space are comparable: 
\[
d(x,y)\leq\pi\abs{x-y}\quad\forall x,y\in\cS^{N}.
\]
Consequently, for every $x\in \cS^N$,
\[
\bE_x \left[d(B_{t},x)^{4}\right]\lesssim \bE_x\left[\abs{\int_0^t \Big( Id-\frac{B_s\tensor B_s}{N} \Big)dW_s }^4\right]+\bE_x\left[\abs{\int\frac{B_s}{N}ds}^4\right]
\lesssim N^2 t^2.
\]
This yields the first item; the second item is then an immediate consequence of the fact that $X_{t}$
solves the SDE \eqref{eq:langevin}, the estimate on $B_{t}$, and the $C$-regularity of $V$ \eqref{eq:Grad-est-U}.
\end{proof}

For $n\geq 1$, let $((F_N^k)_{k\in[n]})_N$ be a sequence of $n$-tuples of smooth functions, 
i.e., for each $N$, $(F_{N}^k)_{k\in [n]}\in C^\infty(\cS^N)^n$. We call such a sequence a \emph{family of observables}.
We seek to study families of observables whose fluctuations are of order 1 in time.

\begin{definition}
A family of observables $((F_N^k)_{k\in [n]})_N$ is \emph{mild} if 
for every $k\in[n]$, there is a $K_k\geq 0$ such that  $F^k_N$ satisfies the Sobolev estimate,
\begin{align}\tag{\textbf{SE}}
\limsup_{N\to\infty}\left(\norm{F_N^k}_{L^\infty(\mathcal S^N)}+ \sqrt{N} \norm{\,|{\nabla F_{N}^{k}}|\,}_{L^\infty(\mathcal S^N)}\right) & \leq K_{k}\,.\label{eq:gradient-estimate-tightness}
\end{align}
\end{definition}

With this notion in hand, we present the main tightness lemma, which is an elementary
application of the Kolmogorov criterion for tightness. We are  concerned with tightness in the spaces $C([0,T])$ endowed with the strong topology and $C([0,T])^n$ endowed with the product topology.
\begin{lemma}
\label{lem:tightness-general-lemma}
Let $X_t$ be Langevin dynamics corresponding to a C-regular sequence $(V_N)$. Let $n\geq1$ and let $( F_{N}^{k}) _{k\leq n}\subset C^{\infty}(\cS^{N})^n$ be a mild family of observables.
Then for all $T\geq0$ and for every sequence $(x_{N})$ with $x_{N}\in\cS^{N}$,
the family $\left( \left(F_{N}^{k}(X_{t})\right)_{k}\right) _{N}$ is
tight in $C\left(\left[0,T\right]\right)^n$.
\end{lemma}
\begin{proof}
As $C([0,T])^{n}$ is equipped with the product topology, it suffices,
by a union bound, to show that for each $k$ and $\epsilon>0$,
there is a compact $E_{\epsilon}^{k}\subset C([0,T])$ such that for every $N$, 
\[
Q_{x_{N}}\left((F_N^k(X_t))_t \in E^{k}_\eps\right)\geq1-\epsilon\,,
\]
For this, it suffices 
to check the following Kolmogorov-type criterion (see, e.g., \cite[Exer.~2.4.2]{stroock1979multidimensional}): there exists $L_k$ such that  for every $x\in \cS^N$, every $t,s\leq T$ and every $N$,
\begin{align*}
\bE_x \left[|{F_{N}^{k}(X_{t})-F_{N}^{k}(X_{s})}|^{4}\right] & \leq L_{k}(t-s)^{2}\,,\\
\|{F_{N}^{k}}\|_{L^{\infty}(\mathcal S^N)} & \leq L_{k}\,,
\end{align*}
This follows by the assumption~\eqref{eq:gradient-estimate-tightness} and \prettyref{lem:brownian-continuity} combined with the Markov property.
\end{proof}

With the above in hand, we can now introduce the main notion of this section,
namely, that of a quasi-autonomous family of observables. 
\begin{definition}
Let $V$ be C-regular. 
A family of observables $(F_N^k)_{k\in[\ell]}$ is \emph{$V$-quasi-autonomous} if there exists $n\geq \ell$ and auxiliary observables $(F_N^k)_{k=\ell+1}^{n}$ such that the following hold.
\begin{enumerate} 
\item The \emph{augmented family}, $(F_N^k)_{k\in[n]}$, is mild. 
\item If $X_t$ is the Langevin dynamics with Hamiltonian $V_N$, then for every $k\leq \ell$, there is a smooth function $\cF_k:\R^n\to\R$ and a bounded stochastic process $g_N^k(t)$ adapted to the filtration of $(B_s)_{s\leq t}$
such that for every $N$,
\begin{equation}{\tag{C1}}\label{eq:splitting}
LF_{N}^{k}(X_{t})=\cF_{k}(F_{N}^{1}(X_t),\ldots,F_N^n(X_t))+g_{N}^{k}(t)\,,
\end{equation}
\noindent (where we recall $L$ is given by~\eqref{eq:langevin-generator-general}) and moreover, the sequence $g_N^k$ satisfies
\begin{equation}\tag{C2}\label{eq:g-decay}
\lim_{N\to\infty}\sup_{x\in\cS^N}\bE_{x}\left[\|{g_{N}^k}\|_{L^\infty([0,T])}\right]=0\,.
\end{equation}
\end{enumerate}
We call the functions, $(\cF_k)_{k\leq \ell}$, the \emph{dynamical functions} of the family.
\end{definition}

We then have the following theorem which is the main result of this section.
\begin{theorem}
\label{thm:differentiability-general-lemma} Let $V$ be $C$-regular. 
Let $\left( F_{N}^{k}\right) _{k\leq\ell}$ be a $V-$quasiautonomous
family of observables with auxiliary observables $(F_N^{k})_{k=\ell+1}^n$ and dynamical functions $(\cF_k)_{k\leq \ell}$.
Finally, let $u_N^k(t)= F_N^k(X_t)$ denote the evolution of the augmented family. Then for every $T\geq 0$, $(u_N^k(t))_i=1^n$ is tight in $C([0,T])^n$.
Furthermore, for any weak limit point of this sequence, $(u^k(t))_{k\leq n}$,
we have that for every $m\leq \ell$, $u^m$ is continuously differentiable and satisfies the integral equation
\begin{equation}
u^m(t) = u^m(0)+ \int_0^t \cF_m(u^1(s),\ldots,u^n(s))ds \,.
\end{equation}
\end{theorem}
\begin{proof}
The tightness follows by the assumption that the augmented family is mild and \prettyref{lem:tightness-general-lemma}.
We now turn to the proof of the integral equations.

By It\^o's lemma and the assumptions, for every $k\leq \ell$, we can write
\begin{align*}
u_{N}^{k}(t) & =u_{N}(0)+\int_{0}^{t}\cF_{k}(u_{N}^{1}(s),...,u_N^n(s))ds+\cE_{N}^{k}(t)\,,\\
\cE_{N}^{k}(t) & =\int_{0}^{t}g_{N}^{k}(s)ds+M^k_{t}\,,
\end{align*}
where $M^k_{t}$ is a martingale with quadratic variation
\[
[M^k]_{t}=\int_{0}^{t}\abs{\nabla F^{k}_N(X_s)}^2ds\,.
\]
Fix any convergent subsequence $(N_i)_i$,
\[
\left(u_{N_{i}}^{k}\right)\convdist\left(u^{k}\right),
\]
where $(u^{k})$ is a weak limit point. If we let 
\[
\bar u_{N}^{k}=u_{N}^{k}-\cE_{N}^{k}\,,
\]
then 
\[
\bar u_{N_{i}}^{k}(t)\convdist \bar u^k(t)=u(0)+\int_{0}^{t}\cF_{k}(u^{1}(s),...,u^k(s))ds\,.
\]

On the other hand, for every $k$, there exists $K_k$ such that 
\[
\bE_x\Big[\|{M^k_{t}}\|_{C([0,T])}\Big]\leq\bE_x\sqrt{\int_{0}^{T}\abs{\nabla F^k_N(X_s)}^{2}ds}\leq\sqrt{\frac{K_kT}{N}}\,,
\]
by Burkholder--Davis--Gundy and the mildness criterion \eqref{eq:gradient-estimate-tightness}.
Thus, by assumption on $g^k_N$,
\[
\limsup_{N\to\infty}\sup_{x\in \mathcal S^N}\bE_x\Big[\|{\bar u_{N}^{k}-u_{N}^{k}}\|_{C([0,T])}\Big]=\limsup_{N\to\infty}\sup_{x\in \mathcal S^N}\bE_x\Big[\|{\cE_{N}^{k}}\|_{C([0,T])}\Big]=0\,.
\]
It then follows, by a standard approximation argument \cite[Thm. 3.29]{Kallenberg02}, that
\[
\left(u_{N_{i}}^{k}\right)\convdist\left(\bar u^{k}\right),
\]
implying the claimed integral inequality.
Since $(\cF_k)_{k\leq \ell}$ are smooth and $(u^{k})_{k=1}^{n}$ are continuous, $\bar u^{k}$ is continuously differentiable
for each $k\leq\ell$. Thus the corresponding $u^{k}$ are as well. 
\end{proof}

\section{Regularity theory for $p$-spin models {and the $\cG^k$ norm}}\label{sec:regularity}
In order to apply the preceding, we need to control the regularity
of the $p$-spin Hamiltonian in the appropriate Sobolev norm.
Throughout our entire discussion of the $p$-spin models, we fix $p>2$ and $H$ will 
always refer to the $p$-spin Hamiltonian, $H=H_{N,p}$ from \prettyref{eq:p-spin-defn}.

Recall the homogeneous Sobolev norm $\dot W^{k,\infty}(\cS^N)$,
\[
\norm{f}_{\dot W^{k,\infty}}:=\norm{ \sqrt{\g{\nabla^k f,\nabla^k f}}}_{L^\infty(\mathcal S^N)}\,.
\]
It will turn out to be easier to work with the following, more controlled, function space.
\begin{definition}\label{def:g-norm}
We say that a function is in the space $\cG_{K}^k(\cS^N)$, if 
\[
\norm{f}_{\cG^k_K} := \sum_{0\leq \ell\leq k} K^{{\ell}/2}\norm{\,|{\nabla^{\ell} f}|_{op}\,}_{L^\infty(\mathcal S^N)} <\infty\,.
\]
Here, $\abs{\nabla^k f}_{op}(x)$ denotes the  natural operator norm when $\nabla^k f$ is viewed as
a $k$-form acting on the $k$-fold product of the tangent space $T_x \cS^N$. Throughout the paper, unless otherwise specified $|\nabla^k f|$ will denote this norm. 
We let $\mathcal G^k(\mathcal S^N) = \cG_N^k(\cS^N)$ denote the special case $K=N$.
\end{definition}
\begin{remark}
{By equivalence of norms in finite dimensional vector spaces, $\cG^k$ 
is the classical $W^{k,\infty}$ space with an equivalent norm. }
We use this norm as opposed to the usual one{, based on Frobenius norms, } for the following reason.
For $k\geq2$, we study operator norms
of random tensors. It is well known that
for operators of this type there is a marked difference in
the scaling of the Frobenius and operator norms in the dimension (see, e.g.,~\cite{Vershynin}). {The choice of $K=N$ compensates for this difference in scaling at every $k$.} {To avoid ambiguity about the choice of Frobenius norm 
as opposed to operator norm, we denote this space here by $\cG^k$ as opposed to $W^{k,\infty}$.}
\end{remark}

The main goal of this section is the following estimate. In the following we say that a function $f(x)$ is of at most polynomial growth if there are $n,C,$ and $c>0$ such that 
$\abs{f(x)}\leq C\abs{x}^p+c$.

\begin{theorem}[Regularity]\label{thm:reg}
For every $k$, there exist $K_1(p,k), K_2(p,k),c_p >0$ such that
$K_1$ is of at most polynomial growth in $p$, and
\begin{enumerate}
\item $H$ is in $\cG^k$ uniformly in $N$ with high probability:
\begin{equation}\label{eq:g-norm-bound}
\limsup_{N\to\infty}\frac{1}{N}\log\prob(\norm{H}_{\cG^k} \geq K_1 N ) \leq -c\,,
\end{equation}

\item $H$ satisfies $\dot W^{k,\infty}$ bounds for $k=1,2$ uniformly in $N$ with high probability:
\begin{equation}\label{eq:sobolev-bound}
\begin{aligned}
\limsup_{N\to\infty}\frac{1}{N}\log\prob(\norm{H}_{\dot W^{k,\infty}}\geq K_2\sqrt{N}) \leq - c\,. 
\end{aligned}
\end{equation}
\end{enumerate}
\end{theorem}

In the present paper, Theorem~\ref{thm:reg} is only applied with $k\leq 4$; however, we prove it in this greater generality as it can be applied to deduce tightness of a much greater family of observables, and proves relevant in other settings, e.g.,~\cite{BAGJ18b}.  
\begin{remark}
We note here that by an application of Borell's inequality, and the preceding estimates,
the same bound also holds for the so called ``mixed $p$-spin models", i.e.,
for Hamiltonians of the form $H=\sum a_p H_p(x)$ where $\sum a^2_p 2^p<\infty$,
except the constants now depend on the sequence $(a_p)_p$ instead of $p$ (and, of course, the polynomial 
dependence is lost).
\end{remark}

It will also be essential to control certain spectral properties of the Hessian. 
Recall that the restriction of the Euclidean Hessian to $T\mathcal S^N$, denoted $G=G_{N,p}$, is distributed as $C_p \cdot M$ where $M$ is drawn from the 
Gaussian orthogonal ensemble (GOE) normalized such that its limiting spectrum is supported on $[-\sqrt 2,\sqrt 2]$ (see, e.g., \cite{ABC13}).
We then have the following uniform estimates on $G$. Let $E_{\mu_{SC}}$ be the expectation with respect to the semicircle law, $\mu_{SC}$,
 supported on $[-\sqrt 2,\sqrt 2]$.

\begin{theorem}\label{thm:GOE-reg}
There exist $f(p,c),g(p,c),h(p,c)>0$ such that for every $\delta>0$, and every $c>0$,
\begin{align}
\mathbb P\bigg( \sup_{x\in \mathcal S^N} \big|\tr G(x) \big| \geq cN^{\frac 12 +\delta}\bigg ) &\lesssim  \exp(-f(c)N^{1+2\delta})\,, \label{eq:tr-goe}\\
\mathbb P\bigg(\sup_{x\in \mathcal S^N} \big|\tr G^2(x)-C_p^2 N E_{\mu_{SC}}[\lambda^2]\big| \geq cN^{\frac 12+\delta}\bigg) &\lesssim \exp(-g(c) N^{1+2\delta})\,, \label{eq:tr-goe-squared}\\
\mathbb P\bigg(\sup_{x \in \mathcal S^N} \abs{\nabla \tr (G(x))}  \geq cN^{\delta}\bigg) &\lesssim \exp(-h(c)N^{1+2\delta})\,. \label{eq:tr-goe-lip}
\end{align}
\end{theorem}

\noindent Observe that in~\eqref{eq:tr-goe-squared}, using that $C_p^2 = 2p(p-1)$ and the moments of $\mu_{SC}$, we have
\begin{align*}
\mathbb E[\tr G^2(\mathbf n)] = C_p^2 N E_{\mu_{SC}} [ \lambda^2] = p(p-1) N\,.
\end{align*}

As the proofs of Theorem~\ref{thm:reg} and~\ref{thm:GOE-reg} are technical and independent of the rest of the paper, readers interested primarily in the main implications of the paper can freely skip to Section~\ref{sec:diff-ineq}.

\subsection{Some facts regarding the geometry of $\cS^N$}
In the rest of this section, we will heavily use certain elementary
facts regarding the geometry of $\cS^N$. We recount them here.

We begin  by observing the following estimate on the covering number of $\cS^N$.
For a set $A\subset\R^N$
equipped with the Euclidean metric, let  $\mathcal N(A, r)$ 
be the minimum number of Euclidean balls of radius $r$ needed to cover $A$. Recall (see, e.g., \cite[Lemma 5.1]{Vershynin})
that
\begin{equation}\label{eq:covering-number-sphere}
\cN( \bS^{N-1}(1), r) \leq \left( \frac{4}{r}\right)^N\,.
\end{equation}
In particular, observe that it takes $C^N \delta_N^{-N/2}$ balls of the form $B_x= \{y:\frac 1N \sum_i x_i y_i \geq 1-\delta_N\}$ to cover the sphere $\mathcal S^N$ for some universal $C>0$.

We now remind the reader of the following results regarding
the differential geometry of $\cS^N$.
Recall that the  Ricci tensor of $\cS^N$ satisfies
\begin{equation}\label{eq:ricci}
\Ric_{\cS_N} = (1-\frac{1}{N})Id\,,
\end{equation}
where $Id$ is the identity on $T\cS^N$.  As a consequence, Bochner's formula in this setting
reads:
for any $u\in C^\infty(\cS^N)$,
\begin{equation}\label{eq:bochner}
\frac{1}{2}\Delta\abs{\nabla u}^2 = \abs{\nabla^2 u}_F^2 + \langle \nabla\Delta u,\nabla u\rangle +(1-\frac{1}{N}) \g{\nabla u,\nabla u}\,.
\end{equation}

The principle curvatures of $\cS^N$ are all equal, 
$\kappa_i=\kappa$, with $\kappa =\pm1/\sqrt{N}$, where the choice of sign depends on the choice of normal vector.
 Let us work with the convention that $\kappa$ is negative (an ``outward'' facing normal).
 Then the second fundamental form satisfies 
\begin{equation}\label{eq:second-fund-form}
II(X,Y)= \frac{1}{\sqrt{N}}\g{X,Y}P\,,
\end{equation}
where $P$ is the ``outward'' normal direction. As $H$ is a degree $p$ homogeneous function, $P H= p H/{\sqrt N} $. Consequently, the Euclidean Hessian of $H$, restricted to $T\cS^N$, can be expressed as 
\begin{equation}\label{eq:g-def}
\begin{aligned}
G(x) &= \nabla^2 H(x) + II(\cdot,\cdot) H(x), \\
&=\nabla^2 H(x) + p\frac{H(x)}{N} Id\,,
\end{aligned}
\end{equation}
where $Id$ is the identity (i.e., the metric tensor). Going further, we note the following representation of higher covariant derivatives of $H$. 
Let $Sym_{k}$ denote the symmetric group of the set $[k]$ and let $\hat \nabla$ be the Euclidean derivative.
\begin{lemma}\label{lem:diff-geo-fact}
Let $f\in C^{\infty}(\cS^{N})$, $x\in\cS^{N}$ and let $\{X_{i}\}_{i=1}^{k}\subset T_{x}\cS^{N}$.
For each $k\geq 2$ and $\ell \leq k$, there exist functions $c_{\ell,k}:Sym_{2k-\ell}\to\Z$
such that
\begin{equation}\label{eq:kth-deriv}
\nabla^{k}f(X_{1},\ldots,X_{k})=\sum_{\ell=1}^{k}\sum_{\sigma\in Sym_{2k-\ell}}c_{\ell,k}(\sigma)A_{\ell,k}(\sigma)\,,
\end{equation}
where 
\[
A_{\ell,k}(\sigma)=\frac{1}{N^{(k-\ell)/2}}\hat{\nabla}^{\ell}f(W_{\sigma(1)}^{\ell,k},\ldots,W_{\sigma(\ell)}^{\ell,k})\cdot\left\langle W_{\sigma(\ell+1)}^{\ell,k},W_{\sigma(\ell+2)}^{\ell,k}\right\rangle \cdots\left\langle W_{\sigma(2k-\ell+1)}^{\ell,k},W_{\sigma(2k-\ell)}^{\ell,k}\right\rangle \,,
\]
and $W^{\ell,k}_i = X_i$ for $i \leq k$ and $W^{\ell,k}_i = P$ for $k+1 \leq i \leq 2k-\ell$. 
\end{lemma}
The proof of Lemma~\ref{lem:diff-geo-fact} goes by induction, using the above facts---as it is cumbersome and not relevant to the rest of the paper, we defer the proof of this representation to Appendix~\ref{app:diff-geo}.

\subsection{Proof of Item 1 of \prettyref{thm:reg}}

Before turning to the proof of Item 1 of Theorem~\ref{thm:reg}, 
note the following estimate regarding injective tensor norms of i.i.d.\ Gaussian tensors; here $g_{i_1,...,i_k}$ are each i.i.d.\ standard Gaussian random variables. 
(In fact, Lemma~\ref{lem:gaussian-tensor-norm-bound} only requires that $g_{i_1,...,i_k}$ have sub-Gaussian tails.)

\begin{lemma}\label{lem:gaussian-tensor-norm-bound}
Let $A=(g_{i_1,\ldots,i_k}/N^{(k-1)/2})_{i_1,...,i_k}$ be an i.i.d.\ Gaussian $k$-tensor on $\R^N$.
For every $k\geq 1$, there exists $C(k)>0$ such that the injective tensor norm of $A$  satisfies
\[
\E\bigg[\sup_{X_1,\ldots,X_k\in \bS^{N-1}(1)} A(X_1,\ldots, X_k)\bigg] \leq C N^{1-(k/2)}
\]
\end{lemma}
\begin{proof}
Let $\Sigma= \Sigma_\epsilon$ be an $\epsilon$-net of $\mathbb S^{N-1}(1)$ for an $\epsilon$ to be determined and let $\Sigma^k$ be its $k$-fold Cartesian product. 
By multi-linearity of $A$ and the triangle inequality, 
\begin{align*}
|A| \leq \sup_{\Sigma^k} A(X_1,...,X_k) + \epsilon k |A|\,,
\end{align*}
where we recall that $|A|$ denotes the operator norm of $A$. Then if $\epsilon \leq 1/(2k)$, we have 
\[
\prob( \abs{A} \geq \lambda)\leq \prob\left(\bigcup_{\Sigma^k} \left\{ \abs{A(X_1,\ldots,X_k)} \geq \lambda/2\right\}\right)\,.
\]
 For any $X_1,\ldots,X_k$ unit vectors, 
\[
\E [ A(X_1,\ldots,X_k)^2 ]=\frac{1}{N^{k-1}}\,,
\]
by definition of $A$, so that by Gaussian concentration,
\[
\prob\left(  \abs{A(X_1,\ldots,X_k)} \geq \lambda N^{1-(k/2)}\right) \leq \exp(-N^{2-k} \lambda^2/(2 N^{1-k})) \leq \exp (- \lambda^2 N/2)\,.
\]
By isotropy and a union bound over $\Sigma^k$ (bounding $|\Sigma^k|$ by~\eqref{eq:covering-number-sphere}), we see that 
\begin{align*}
\mathbb P(\abs{A} \geq \lambda N^{1-(k/2)}) \leq k(4\epsilon^{-1})^{N} e^{-\lambda^2 N/8}\,.
\end{align*}
The upper bound on $\mathbb E[ \abs{A}]$ then follows by integration.
\end{proof}

With these preliminaries in hand, we are in position to prove Theorem~\ref{thm:reg}.
 Observe that $H$ is almost surely smooth and for $x,y\in \cS^N$,
\begin{align*}
\mathbb E[H(x)H(y)]=N R^p (x,y)
\end{align*}
where $R(x,y)=\frac 1N \sum_{i} x_i y_i$ is the normalized overlap of $x$ and $y$.
\color{black}

\subsection*{\textbf{{Proof of Theorem~\ref{thm:reg} }}}
Since the Frobenius norm and operator norm of a matrix 
always satisfy $\abs{A}_F\leq \sqrt{N}\abs{A}$ for any $A$, the second item is an immediate consequence of the first.

The proof of the first item is in two steps. 
First, we bound the operator norm of the Euclidean derivatives (denoted $\hat \nabla ^k$) 
viewed as operators on the bundle $T\cS^N\oplus N\cS^N$ (where $N\cS^N$ is the normal bundle), via standard Gaussian comparison inequalities. 
We then prove the bounds on the $\cG^k$ norms for all $k$, inductively using the corresponding bounds on the Euclidean derivatives.

\begin{step}
For every $k$, there exists $K_0(p,k)>0$ and $c(p,k,K)>0$ such that for every $K> K_0$,  
\begin{equation}\label{eq:euclidean-deriv-bound}
\mathbb P\bigg(\sup_{x\in \mathcal S^N} \sup_{\substack{X_1,\ldots,X_k\in T_x \R^N\\ \norm{X_i}=1} } \hat{\nabla}^kH(x)(X_1,\ldots,X_k) \geq K N^{1-k/2}\bigg) \lesssim \exp(- cN)\,.
\end{equation}
In particular, $K_0(p,k)$ has at most polynomial growth in $p$.
\end{step}
\begin{proof}
Clearly it suffices to consider $k\leq p$. Let $E=\cS^N\times (\bS^{N-1}(1))^k$ and consider the field $\psi:E\to\R$ defined by
\[
\psi(x,X_1,\ldots,X_k) = \frac{1}{N^{(p-1)/2}} \sum_{i_1,...,i_p} J_{i_1,\ldots,i_p} X^{i_1}_1\cdots X^{i_k}_k\cdot x_{i_{k+1}}\cdots x_{i_p}\,.
\]
Observe that as Gaussian processes on $E$,
\[
 \hat\nabla^kH(x)(X_1,\ldots,X_k)\eqdist C_{p,k} \psi(x,X_1,\ldots X_k)\,,
\]
for some constants $C_{p,k}$ which are of at most polynomial growth in $p$.
Thus it suffices to estimate  $\sup_E \psi$, which we do as follows.

The covariance of $\psi$ satisfies for every $(x,X_1,...,X_k), (x',X_1',...,X_k') \in E$,
\begin{equation}
\begin{aligned}
\cov(\psi(x,X_1,\ldots,X_k),\psi(x',X'_1,\ldots, X'_k)) &= \E \big[\psi(x,X_1,\ldots,X_k)\psi(x',X'_1,\ldots,X'_k)\big] \\
&= \frac{1}{N^{k-1}} R(x,x')^{p-k} \prod_i  X_i\cdot X_i',
\end{aligned}
\end{equation}
where $X\cdot Y$ is the usual Euclidean inner product. 
Evidently, this is uniformly bounded on all of $E^2$ by $N^{-(k-1)}$. Consequently, Borell's inequality implies that
\begin{align}\label{eq:borell-psi}
\frac{1}{N}\log\prob\Big(\mbox{$\sup_E$} \psi \geq \E \big[\mbox{$\sup_{E}$} \psi\big] + u\Big)\leq - u^2/(2N^{-(k-1)})\,.
\end{align}
It remains to bound the expectation $\mathbb E[\sup_{E} \psi]$.

This will follow by Gaussian comparison. Define $\phi_1: \mathcal S^N \to \mathbb R$ and $\phi_2: (\mathbb S^{N-1}(1)) ^k \to \mathbb R$, by
\begin{align*}
\phi_1(x) &= \frac{1}{N^{k/2}} H_{p-k}(x)\,,\\
\phi_2(X_1,\ldots,X_k) &= \frac{1}{N^{(k-1)/2}}\sum_{i_1,\ldots,i_k} g_{i_1,\ldots,i_k} X_1^{i_1}\cdots X_k^{i_k}\,,
\end{align*}
where $H_{p-k}$ is a $(p-k)$-spin Hamiltonian and $(g_{i_1,\ldots,i_k})_{i_1,...,i_k}$ are i.i.d. standard Gaussians.
Observe that, if we let $\mathbf{X} = (X_1,\ldots X_k)$ and $\mathbf X' = (X_1',...,X_k')$ in $ (\bS^{N-1}(1))^k$, 
\begin{align*}
\E \big[(\psi(x,\mathbf{X})-\psi(x',\mathbf{X}'))^2\big] &= \frac{1}{N^{k-1}}\left[  R(x,x)^{p-k} + R(x',x')^{p-k} - 2R(x,x')^{p-k}\prod X_i\cdot X_i'\right]\\
&\leq  \E\big[(\phi_1(x)-\phi_1(x'))^2\big] + \E \big[(\phi_2(\mathbf{X})-\phi_2(\mathbf{X}'))^2\big] \\
& = \mathbb E[ \big( (\phi_1(x)+ \phi_2 (\mathbf X)) - (\phi_1(x')+ \phi_2(\mathbf X'))\big)^2]\,.
\end{align*}
Thus by the Sudakov-Fernique inequality~\cite{LedouxTalagrand},
\begin{align*}
\mathbb E\Big[\sup_{(x,\mathbf X)\in E} \psi(x,\mathbf X)\Big] \leq \mathbb E\Big[ \sup_{(x,\mathbf X) \in E} \phi_1(x)+ \phi_2 (\mathbf X)\Big]\leq \mathbb E\Big[ \sup_{x\in \mathcal S^N} \phi_1(x)\Big] + \mathbb E\Big[ \sup_{\mathbf X \in (\mathbb S^{N-1}(1))^k} \phi_2(\mathbf X)\Big]\,.
\end{align*}
It thus suffices to separately bound these two expectations. 
By Dudley's entropy bound \cite{LedouxTalagrand} and \eqref{eq:covering-number-sphere},
the ground state of the $(p-k)$-spin glass is order $N$, so $\phi_1$ satisfies
\[\E [\mbox{$\sup_{\mathcal S^N}$} \phi_1] \leq C(p) N^{1-(k/2)},\]
where $C(p)$ is of at most polynomial growth in $p$ (for a more precise bound see, e.g., \cite{ABC13}).
The estimate on $\mathbb E[\sup \phi_2]$ follows from \prettyref{lem:gaussian-tensor-norm-bound}.
The result then follows from~\eqref{eq:borell-psi}.
\end{proof}

We now turn to the proof of the main bounds we desire on the sphere.
\begin{step}
It suffices to show the following inductively. For every $k$, there exists a $K_0(p,k)>0$ of at most polynomial growth in $p$ and a $c(p,k,K)>0$ such that for every $K> K_0$, 
\begin{equation}
\mathbb P\bigg(\sup_{x\in \mathcal S^N} \sup_{X_1,\ldots,X_k\in (\bS^{N-1}(1))^k} \nabla^k H (X_1,\ldots,X_k) \geq K N^{1-k/2}\bigg)\lesssim \exp(-cN)\,.
\end{equation}
\end{step}

\begin{proof} 

To bound the case $k=1$, it suffices to bound the operator
norm of $\hat \nabla H$ as an operator on $T\R^N$. Thus that case is complete by \prettyref{eq:euclidean-deriv-bound} with $k=1$.
For every $k\geq 2$, by Lemma~\ref{lem:diff-geo-fact}, it suffices to show that for each $\ell \leq k$, we have the bound 
\begin{align*}
\sup_{X_1,...,X_k:\|X_i\|=1} \sup_{\sigma \in Sym_{2k-\ell}} A_{\ell, k}(\sigma) \leq K(p,k) N^{1-k/2},
\end{align*}
For some $K$ of at most polynomial growth in $p$. This follows from the fact that each of the inner products in $A_{\ell,k}(\sigma)$ are bounded by $1$, and the $\ell$'th Euclidean derivative part of $A_{\ell,k}(\sigma)$ is bounded by $K\cdot N^{1-\ell/2}$ as per~\eqref{eq:euclidean-deriv-bound}; Consequently, each summand in the expansion~\eqref{eq:kth-deriv} is at most some $C_{\ell,k,p} N^{1-k/2}$, implying the desired.
\end{proof}

\subsection{Proof of Theorem~\ref{thm:GOE-reg}}

The proof of Theorem~\ref{thm:GOE-reg} relies on the following general lemma that boosts 
a pointwise concentration estimate to a uniform bound on the sphere via a minimal modulus of continuity bound. {Since we consider isotropic fields, we let $\mathbf n=(\sqrt N,0,...,0)$ be the north pole and stand in as a fixed point on $\cS^N$.}

\begin{lemma}\label{lem:sphere-union-bound}
Suppose that $F$ is an isotropic random field on $\mathcal S^N$ that is $\mathbb P$--a.s.\ continuous and satisfies the following:
\begin{enumerate}
\item There exists $c_1(c)>0$ such that for every $c,\delta>0$,
\begin{align*}
\mathbb P(|F(\mathbf n) - \mathbb E[F(\mathbf n)]| \geq cN^{\frac 12+\delta}) \lesssim \exp(-c_1 N^{1+2\delta})\,.
\end{align*}
\item There exists an $\alpha>0$  and $c_2(c)>0$ such that for every $c,\delta>0$, 
\begin{align*}
\mathbb P\bigg(\sup_{x,y: R(x,y)\geq\color{black} 1-N^{-\alpha}} |F(x)-F(y)| \geq cN^{\frac 12 +\delta}\bigg) \lesssim \exp(-c_2 N^{1+2\delta})\,.
\end{align*}
\end{enumerate}
Then there exists $f(c)>0$ depending on $c_1,c_2$ such that for every $c,\delta>0$ 
\begin{align*}
\mathbb P\bigg( \sup_{x\in \mathcal S^N} |F(x)-\mathbb E[F(\mathbf n)] | \geq cN^{\frac 12+\delta} \bigg) \lesssim \exp(-f(c) N^{1+2\delta})\,.
\end{align*}

\noindent Similarly, if $F$ is almost surely nonnegative and $\mathbb E[F(\mathbf n)] = O(\sqrt{N})$, if we replace condition 2 with
\begin{itemize}
\item[\emph{(2')}] There exists an $\alpha>0$  and $c_2(c)>0$ such that for every $c,\delta>0$, 
 \[\mathbb P\left(\sup_{x,y:R(x,y)\geq 1-\delta_N} |F^2(x)-F^2(y)| \geq c N^{1 +2\delta}\right) \lesssim \exp(-c_2 N^{1+2\delta})\,,\]
 \end{itemize} 
then there exists $f(c)>0$ depending on $c_1,c_2$ such that for every $c,\delta>0$, 
\begin{align*}
\mathbb P\bigg( \sup_{x\in \mathcal S^N} F(x) \geq  c N^{\frac 12+\delta} \bigg) \lesssim \exp(-f(c) N^{1+2\delta})\,.
\end{align*}
\end{lemma}

\begin{proof}
Let $\eta_N = N^{-\alpha}$ for the $\alpha$ given by condition 2.\ Then by \eqref{eq:covering-number-sphere}, there exists a set of points $x_1,...,x_K\in \mathcal S^N$ such that the balls $B_{i} = \{y\in \cS^N: R(x_i,y)\geq 1-\eta_N\}$ cover $\mathcal S^N$ and for some universal $C,c_\star>0$,
\begin{align*}
K \leq C^N N^{\alpha N/2}\leq \exp (c_\star N \log N)\,.
\end{align*} 
For each $i$, condition (1) holds with $\mathbf n$ replaced by $x_i$  by isotropy. Fix any $c>0$ and $\delta>0$; for every $1\leq i \leq K$, define the events 
\begin{align*}
\Gamma^1 &  = \bigg\{ \sup_{x,y:R(x,y)\geq 1-\eta_N} |F(x)-F(y)| <\tfrac 12 cN^{\frac 12 +\delta}\bigg\}\,, \\
\Gamma_i^2 & = \big\{ |F(x_i)-\mathbb E[F(\mathbf n)]| < \tfrac 12 cN^{\frac 12+\delta}\big\}\,.
\end{align*}
By the choice of $\{x_i\}$, 
\begin{align*}
\Big (\Gamma^1 \cap \bigcap_{i\leq K} \Gamma_i^2\Big) \subset \Big\{\sup_{x\in \mathcal S^N} |F(x)-\mathbb E[F(\mathbf n)]| < cN^{\frac 12+\delta}\Big\}\,.
\end{align*}
As a result, by isotropy, a union bound over $1 \leq i \leq K$, and conditions 1--2, for $c_1 = c_1(c/2)$ and $c_2 = c_2(c/2)$, we obtain for every $\delta>0$, 
\begin{align*}
\mathbb P\Big(\sup_{x\in \mathcal S^N} |F(x)-\mathbb E[F(\mathbf n)]| \geq cN^{\frac 12 +\delta}\Big) & \leq \mathbb P\big((\Gamma^1)^c\big)+ K\mathbb P\big((\Gamma_1^2)^c\big)\\ 
& \lesssim \exp \big(-\tfrac 12 [c_1 \wedge c_2] N^{1+2\delta}\big)\,,
\end{align*}
yielding the first inequality. When condition 2 is replaced with condition 2', 
we let
\begin{align*}
\Gamma^1 = \Big\{ \sup_{x,y:R(x,y)\geq 1-\eta_N} |F^2(x)-F^2(y)| < \tfrac 12 c^2 N^{1+2\delta} \Big\}\,, \\
\Gamma_i^2 = \big\{ F(x_i) < \tfrac 1{\sqrt 2} c N^{\frac 12 +\delta}\big\} = \big\{F^2(x_i) <\tfrac 12 c^2 N^{1+2\delta} \big\}\,,
\end{align*}
and obtain
\begin{align*}
\mathbb P\Big(\sup_{x\in \mathcal S^N} F^2 (x) \geq c^2 N^{1+2\delta}\Big) \leq \mathbb P\big((\Gamma^1)^c\big) + K\mathbb P\big((\Gamma^2_1)^c \big).
\end{align*}
The first probability on the right-hand side is bounded by condition 2'. The second probability is bounded by condition 1 
combined with the bound on $\E[F(\mathbf n)]$. 
The conclusion then follows. 
\end{proof}

Recall the Euclidean Hessian $G$ given in~\eqref{eq:g-def}. 
We now wish to verify that condition (1) in Lemma~\ref{lem:sphere-union-bound} holds for $\tr G, \tr G^2$ and $\abs{\nabla \tr G}$.

\begin{lemma}\label{lem:GOE-at-a-point}
For $F$ given by each of $\tr G$, $\tr G^2$, and $\sqrt N \abs{\nabla \tr G}$, there exists $c_1(c)>0$ such that for every $c,\delta>0$, 
\begin{align*}
\mathbb P(|F(\mathbf n) - \mathbb E [F(\mathbf n)]| \geq cN^{\frac 12+\delta}) \lesssim \exp(-c_1 N^{1+2\delta})\,.
\end{align*}
\end{lemma}

\begin{proof}
As $G(\mathbf n)$ is a GOE, up to a constant factor,  $G^2(\mathbf n)$ is a Wishart matrix. 
The results for $\tr G(\mathbf n)$ and $\tr G^2(\mathbf n)$, then follow by standard
concentration estimates for spectral statistics \cite[Corr. 1.6, Corr 1.8]{GuiZeit00} or sums of i.i.d.\ random variables  \cite{LedouxTalagrand}.

Now consider $F(\mathbf n)= \sqrt N \abs{\nabla \tr G(\mathbf n)}$.
Since  $G$ is the restriction of the Euclidean Hessian of $H$ to $T\cS^N$,
 an explicit calculation yields 
 \[
\E \left[E_i \tr G(\mathbf n)E_j \tr G(\mathbf n) \right] = c_p \delta_{ij}\,,
\]
for some $c_p>0$ independent of $N$, where $\{E_i\}$ are an orthonormal frame for $T_\bn\cS^N$. (See, e.g.,~\cite{SubGibbs16} for similar
calculations.) In particular, $\nabla \tr G(\bn)$ is
 a standard Gaussian vector with independent entries, up to a constant factor depending at most on $p$. The result then follows
by concentration of norms of Gaussian vectors \cite{LedouxTalagrand}.
\end{proof}

We now prove the uniform continuity required to obtain conditions 2--2' in Lemma~\ref{lem:sphere-union-bound}.
\begin{lemma}\label{lem:GOE-mod-of-continuity}
Let $F$ be given by $\tr G, \tr G^2$ and $N \abs{\nabla \tr G}^2$. There exists $\alpha>0$ such that for $\eta_N \asymp N^{-\alpha}$,  there exists $c_2(c)>0$ such that for every $c,\delta>0$, we have 
\begin{align*}
\mathbb P\Big( \sup_{x,y: R(x,y)\geq 1-\eta_N} |F(x)-F(y)| \geq cN^{\frac 12+\delta}\Big) \lesssim \exp(-c_2 N^{1+2\delta})\,.
\end{align*}
\end{lemma}

\begin{proof}
We prove the desired for 
 $\sqrt{N} \abs{\nabla \tr G}$. The cases $\tr G$ and $\tr G^2$ follow by the same argument  (and are in fact implicitly proved in what follows). 
For every $x\in \cS^N$, $X\in T_x\cS^N$ with $\abs{X}=1$,
\begin{align*}
\nabla_X\abs{\nabla \tr G}^2 =2\nabla^2(\tr G)(\nabla\tr G,X)\leq 2\abs{\nabla^2\tr G} \abs{\nabla\tr G}.
\end{align*}
Since contractions commute with covariant derivatives
it follows, by \eqref{eq:g-def}, that
\[
\big|{\nabla^k \tr G}\big| = \big|{\nabla^k \tr \nabla^2 H+p(1-\frac{1}{N})\nabla^{k} H}\big|  \leq N\big|{\nabla^{k+2} H}\big| + p\big|{\nabla^k H}\big|\,.
\]
Combining this for $k=1,2$ with the $\cG$-norm estimate \prettyref{eq:g-norm-bound}
and Borell's inequality, we see that
\[
\frac{1}{N^{1+2\delta}}\prob\left(\big\|{\,|{\nabla^k \tr G}|\,}\big\|_{L^\infty}\geq N^{\frac{2-k}2+\delta},\, k=1,2\right) <-c\,,
\]
for some $c>0$ for all $N$ sufficiently large. This implies that 
\[
\frac{1}{N^{1+2\delta}}\log\prob\bigg( \big\|{\,\big|{\nabla\abs{\nabla \tr G}^2}|\,}\big\|_{L^\infty}\geq N^{\frac 12+2\delta}\bigg)< -c\,,
\] 
holds for all $N$ sufficiently large. Choosing $\alpha=1+\delta$ yields the result. %
\end{proof}

\begin{proof}[\textbf{\emph{Proof of Theorem~\ref{thm:GOE-reg}}}]
By Lemma~\ref{lem:sphere-union-bound}, it first suffices to prove that conditions 1--2 hold for $\tr G$ and $\tr G^2$ as these are both isotropic, almost surely continuous fields on $\mathcal S^N$. For these two fields, condition 1 follows from Lemma~\ref{lem:GOE-at-a-point} and condition 2 follows from Lemma~\ref{lem:GOE-mod-of-continuity}. For $\sqrt N |\nabla \tr G(x)|$, it is again an isotropic, almost surely continuous field on $\mathcal S^N$ and we observed 
in the proof of Lemma~\ref{lem:GOE-at-a-point} earlier that $\mathbb E[\sqrt N |\nabla \tr G(\mathbf n)|] = O(\sqrt N)$. Condition 1 for this field follows from Lemma~\ref{lem:GOE-at-a-point}, and condition~2' follows from Lemma~\ref{lem:GOE-mod-of-continuity}.
\end{proof}

\section{Tightness and differentiability of observables}\label{sec:diff-ineq}

We begin our analysis of dynamics for $p$-spin models in this section. Here $X_{t}$ will be Langevin dynamics with Hamiltonian $H=H_{N,p}$. Our main
goal will be to conclude the tightness and differentiability of weak
limit points for the family of observables
\begin{align*}
u_{N}(t) & =-\frac{H(X_t)}{N}\,,
&&v_{N}(t)  =\frac{\abs{\nabla H(X_t)}^{2}}{N}\,,
\end{align*}
and the auxiliary  observable
\[
w_{N}(t)=\frac{1}{N}G(\nabla H,\nabla H)(X_{t})\,.
\]
 Let $\sL_{x_{N},N}^{T}$, a probability measure on $C([0,T])^{3}$,
be the law of the augmented family $(u_{N},v_{N},w_{N})$ started from $X_0 =x_N$, and let
\begin{align}
\cF_1(x,y) & =-px+\beta y\,, \label{eq:udot}\\
\cF_2(x,y,z) & =2p(p-1)+2p^{2}x^{2}-2(p-1)y-2p\beta xy -2\beta z\,.
\end{align}
We then have the following theorem. 
\begin{theorem}\label{thm:p-spin-tightness}
$\prob$-almost surely, the following holds. For every sequence $x_{N}\in\cS^{N}$, the family $(u_{N},v_{N},w_{N})$
is tight. That is $\sL_{x_{N},N}^{T}$ is precompact in the narrow
topology. Furthermore, for any limit point $\sL$, the family $(u,v,w)\sim \sL$ satisfies the integral equations \
\begin{align}
u(t) &=u(0) +\int_0^t \cF_1(u(s),v(s))ds\,, \label{eq:u-integral-equation}\\
v(t) &=v(0)+ \int_0^t\cF_2(u(s),v(s),w(s))ds\,.\label{eq:v-integral-equation}
\end{align}
and the family $(u(s),v(s))_s$ are continuously differentiable.
\end{theorem}

We first need the following estimates on $\Delta H_N$ and $\Delta |\nabla H_N|^2$ to estimate $Lu_N$ and $Lv_N$.

\begin{lemma}\label{lem:Lgrad-squared-estimate}
For every $p>2$, there exists $f(\cdot)>0$ such that  for every $c,\delta>0$, 
\begin{align}
\prob\big(\norm{\Delta H + pH}_{L^\infty(\mathcal S^N)} \geq c N^{1/2+\delta}\big)& \lesssim \exp(-f(c)N^{1+2\delta})\,,\label{eq:laplacian-estimate}\\
\prob\left(\norm{\Delta\abs{\nabla H}^2-2A}_{L^\infty(\mathcal S^N)} \geq c N^{1/2+\delta}\right)& \lesssim \exp(-f(c)N^{1+2\delta})\,,\label{eq:laplacian-grad-estimate}
\end{align}
where
\begin{equation}\label{eq:A-def}
A=Np(p-1)+p^2\frac{H^2}{N}-(p-1)\abs{\nabla H}^2\,.
\end{equation} 
\end{lemma}
\begin{proof}
We begin with the first estimate. 
Taking the trace of \eqref{eq:g-def}, we see that
\[
\Delta H ={-p(1-\frac{1}{N}) H} + \tr G(x)\,.
\]
The result then follows by \eqref{eq:tr-goe} and~\eqref{eq:g-norm-bound}. 

Now for the second estimate, by the Ricci bound \eqref{eq:ricci} and Bochner's formula \eqref{eq:bochner},
\begin{align*}
\frac{1}{2}\Delta \abs{\nabla H}^2 &=\tr(G^2)+(1-\frac{1}{N})p^2\frac{H^2}{N}-2p \tr(G)\frac{H}{N}-(p-1)(1-\frac{1}{N})\abs{\nabla H}^2 
+\g{\nabla \tr G,\nabla H}\,.
\end{align*}
The result then follows upon applying \eqref{eq:tr-goe-squared} to the first term,~\eqref{eq:tr-goe} to the third term, and Cauchy--Schwarz along with~\eqref{eq:tr-goe-lip} to the fifth term.
\end{proof}

\begin{proof}[\textbf{\emph{Proof of \prettyref{thm:p-spin-tightness}}}]
The result will follow by an application of \prettyref{thm:differentiability-general-lemma}.
We begin by observing that $H$ satisfies the gradient estimate  \prettyref{eq:Grad-est-U} by Theorem~\ref{thm:reg}.
Thus $H$ is $\mathbb P$--eventually almost surely $K_1(p,k)$-regular for any fixed $k$.

We now check that these observables are mild. Let 
\[
F_N^{1}=-\frac{H}{N}\,,\qquad F_N^{2}=\frac{\abs{\nabla H}^{2}}{N}\,,\qquad F_N^{3}=\frac{1}{N}G(\nabla H,\nabla H)\,.
\]
Observe that 
\[
\nabla_{X}G(\nabla H,\nabla H)=\nabla G(X,\nabla H,\nabla H)+2G(\nabla^{2}H(\,\cdot\,,X),\nabla H)\,,
\]
so that
\[
\abs{\nabla G(\nabla H,\nabla H)}\leq\abs{\nabla G}\abs{\nabla H}^{2}+2\abs{\nabla^{2}H}\abs{G}\abs{\nabla H}.
\]
Thus by the $\cG^{3}$-norm bound from \prettyref{thm:reg}, and the definition of $G$,
there is a $K_{p}$ such that eventually $\prob$-almost surely,
\[
\norm{F_N^{k}}_{\cG^{1}}\leq K_{p}\qquad\forall k \leq 3.
\]
Thus the Sobolev estimates \prettyref{eq:gradient-estimate-tightness} hold.
Consequently, this family is mild.

To show that $\{F_N^k\}_{k=1}^2$ are quasiautonomous, let 
\begin{align}\label{eq:g-for-p-spin}
g^{1}_N(t) & =\frac{1}{N}\left(\Delta H(X_{t})+pH(X_{t})\right)\,,\\
g^{2}_N(t) & =\frac{1}{N}\left(\frac{1}{2}\Delta\abs{\nabla H}^{2}-A\right)(X_t)\,,
\end{align}
where $A$ is as in \eqref{eq:A-def}. The processes $g^1,g^2$ satisfy \eqref{eq:g-decay} by \prettyref{lem:Lgrad-squared-estimate}.
Recalling the generator $L$ of the Langevin dynamics and writing out $LF_N^k$ using the identity,
\begin{align*}
\langle \nabla H, \nabla |\nabla H|^2 \rangle =  2 \nabla^2 H (\nabla H,\nabla H) = 2 G(\nabla H,\nabla H) -2 p\frac H{N} |\nabla H|^2\,,
\end{align*}
we see we have the splitting \eqref{eq:splitting}.  In particular, the family $\{F^k_N\}_{k=1}^2$ is $H$--quasiautonomous
with dynamical functions $\{\cF_k\}_{k=1}^2$ and auxiliary function $F^3_N$. Thus the conditions of
 \prettyref{thm:differentiability-general-lemma} are satisfied. Consequently the integral equations hold and
any weak limit point is such that $(u,v)$ are continuously differentiable. 
\end{proof}

\subsection{A differential inequality for the evolution of $(u(t),v(t))$}

We end this section by observing the following consequence of
\prettyref{thm:p-spin-tightness}.
Recall  $\Lambda_p>0$ from~\eqref{eq:Lambda-p-defn}, and recall from~\eqref{eq:cf-def}, 
\begin{align}
\cF_1 (u,v) & = - pu + \beta v\,, \nonumber \\
\cF_{2}^L(u,v) &= 2p(p-1) -2(p-1) v + 2pu(pu-\beta v) - 2\beta\Lambda_p v\label{eq:vdot-lb}\,, \\
\cF_{2}^U(u,v) &= 2p(p-1)-2(p-1)v + 2pu(pu-\beta v)+2\beta \Lambda_p v\label{eq:vdot-ub}\,,
\end{align}
Moreover, define the maps
\begin{equation}\label{eq:Phi-def}
\Phi_{L}(u,v)=\left(\begin{array}{c}
\cF_1\\
\cF_{2}^L
\end{array}\right)\,, \qquad\Phi_{U}(u,v)=\left(\begin{array}{c}
\cF_1\\
\cF_2^{U}
\end{array}\right)\,,
\end{equation}
so that for $(x_0,y_0)\in \mathbb R\times \mathbb R_+$, $A_L(t)$ is the solution to the initial value problem $\dot A_L(t) = \Phi_L(A_L(t))$ and $A_L(0)=(x_0,y_0)$, and $A_U(t)$ is similarly defined. 

Also, recall that 
for any sequence $x_N \in \cS^N$, we let $\sU((x_N))$ denote the set of all possible limiting laws of $(u_N(t),v_N(t))_{t\in [0,T]}$ started from $X_0=x_N$, and suppress the dependence on $T$ in the notation. 
We are then able to conclude the proof of the differential inequality. Here and in the following, inequalities for vectors
are to be interpreted as holding in both coordinates simultaneously.


\begin{proof}[\textbf{\emph{Proof of Theorem~\ref{thm:bounding-flows-thm}}}]
For every $T$ and every sequence of initial data $(x_N)$, the tightness of the laws of $(U_N(t))_{t\in [0,T]}$ was established in Theorem~\ref{thm:p-spin-tightness}.  Let $\mathscr U((x_N))$ be as in Definition~\ref{def:s-U}, which by the tightness must be non-empty, fix any $\mathscr L\in \mathscr U((x_N))$, and let $(U(t))_{t\in [0,T]}=(u(t),v(t))_{t\in [0,T]}$ be distributed according to $\mathscr L$. By \prettyref{thm:p-spin-tightness}, $U$ is $\mathscr L$-almost surely differentiable.

By \eqref{eq:u-integral-equation}, $\mathbb P$-almost surely, $\dot{u} = \cF_1(u,v)$ $\sL$-a.s.
By \eqref{eq:Lambda-p-defn},
\[
-\Lambda_p v \leq w \leq \Lambda_p v\,,
\]
$\sL$-a.s. Consequently, $\sL$-a.s.,
\[
\cF_{2}^L(u,v)\leq \cF_2(u,v,w)\leq \cF_{2}^U(u,v)\,,
\]
for all $t\geq 0$.
This combined with \eqref{eq:v-integral-equation}, 
implies that $\mathbb P$--almost surely, for every $T$, every sequence $x_N\in \cS^N$ and every $\sL\in \sU((x_N))$, if $U(t)\sim \mathscr L$,
\begin{equation*}
\sL\left(\Phi_L(U)\leq \dot{U} \leq \Phi_U(U)\right) = 1\,. \qedhere
\end{equation*}
\end{proof}


%

\section{Comparison theory for limiting dynamics}\label{sec:comparison-theory}
Before turning to the proof of Theorem~\ref{thm:full-phase-portrait}, 
we will first understand some general consequences of Theorem~\ref{thm:bounding-flows-thm}
for trajectories of the limiting dynamics $(u(t),v(t))$ (see Figure~\ref{fig:trajectory-comparison}).

\subsection{A trajectory-wise comparison}
Our main result in this section is a means to use the integral curves of two ``bounding flows''
to bound that of our dynamics. Informally, the goal is to confine integral curves of dynamical systems satisfying the differential inequality of Theorem~\ref{thm:bounding-flows-thm} by the integral curves of the lower and upper bounding dynamical systems. 

We wish to compare any weak limit, $U\sim \sL$, of $(u_N(t),v_N(t))$ to these systems. 
To show this comparison result, we first observe the following basic
fact from calculus. 
Recall the notation $v\leq_{2}w$ if $v_{1}=w_{1}$
and $v_{2}\leq w_{2}$ and define $\geq_{2}$ similarly. 
\begin{lemma}
\label{lem:graph-comparison}Let $X,Y:[0,T]\to\R^{2}$ be $C^{1}$.
Let $\Psi:\R^{2}\to\R^{2}$ be $C^{1}$ be such that $\dot{Y}=\Psi(Y)$.
Let $E_{+}=\{\Psi_{1}>0\}$ and $E_{-}=\left\{ \Psi_{1}<0\right\} .$
We then have the following. Suppose that
\begin{itemize}
\item $X(0)=Y(0)\in E_{+}\cup E_{-}$ ,
\item For $t>0$, if $X(0)\in E_+$, then $\Psi_{1}(X),$$\Psi_{1}(Y)>0$, and if $X(0)\in E_+$ then $\Psi_1(X),\Psi_1(Y)<0$,
\item $X$ satisfies the differential inequality $\dot{X}\leq_{2}\Psi(X)$
 (resp. $\dot{X}\ge_{2}\Psi(X)$).
\end{itemize}
Then the functions $f_{X}(u)=X_{2}\circ X_{1}^{-1}(u)$ and $f_{Y}(u)=Y_{2}\circ Y_{1}^{-1}(u)$
are well-defined, differentiable for $u\neq X_{1}(0)$, and $f_{X}(u)\leq f_{Y}(u)$
$($resp. $f_{X}(u)\geq f_{Y}(u))$ on the intersection of their domains. 
\end{lemma}
\begin{proof}
Without loss of generality, take $X(0)=Y(0)=(0,0)$. Furthermore, it suffices
to consider the a case where $(0,0)\in E_+$ and $\dot X\leq_2 \Psi(X)$, as the others are identical. By the positivity
assumption for $\Psi_1(X)$ and $\Psi_1(Y)$, $X_{1}(t)$ and $Y_{1}(t)$ are invertible so that the
functions $f_{X}$ and $f_{Y}$ are well-defined. Furthermore, the
inverse functions are differentiable for $u>0$. It remains to show
the inequality between $f_X$ and $f_Y$.

Suppose that there is some $w$ for which $f_{X}(w)>f_{Y}(w).$ Observe
that for every $u>0$ for which $f_{X}(u)=f_{Y}(u)$, we have

\begin{equation}
f_{X}'(u)\leq\frac{\Psi_{2}\left(u,f_{X}(u)\right)}{\Psi_{1}(u,f_{X}(u))}=f_{Y}'(u)\,,\label{eq:graph-eq}
\end{equation}
by the third assumption. Thus $w>0$. On the other hand, let $v=\sup\{u\leq w:f_{X}(u)=f_{Y}(u)\}$.
Then by continuity of $f_{X},f_Y$ and maximality of $v$, we have that
$f_{X}(u)>f_{Y}(u)$ for all $u\in (v,w)$. On the other hand, by continuity
of $f_{X},f_Y$ and \eqref{eq:graph-eq}, we have that $f_{X}'(v)\leq f_{Y}'(v)$.
This is a contradiction. 
\end{proof}

In the following, we use the convention that $x>\infty$ means $x=\infty$.
We can deduce the following corollary of the preceding. Recall from Section~\ref{subsec:phase-portrait-defs}, the definitions of $A_L(t), A_U(t)$, the graphs of their trajectories, $\gamma_L(u; (x_0,y_0))$ and $\gamma_U(u,(x_0,y_0))$, and the sets $W_-$ and $W_+$ defined in~\eqref{eq:W-pm}.  

\begin{corollary}
\label{cor:graph-cor} Suppose that $X:[0,T]\to \mathbb R^2$ is $C^1$ and satisfies \condI . Then,
\begin{enumerate}
\item If $X(0)\in W_{\pm}$, then $(X(t))_{t\leq \tau_{W_\pm^c}}$ is the graph of a
function $\gamma(\,\cdot\,;X(0))$.
\item The domains of $\gamma,\gamma_{L},$ and $\gamma_{U}$ satisfy 
\[\mbox{Dom}(\gamma_{L}(\,\cdot\,;X(0)))\subseteq \mbox{Dom}(\gamma(\,\cdot\,;X(0)))\subseteq Dom(\gamma_{U}(\,\cdot\,;X(0)))\,,\]
when $X(0)\in W_{+}$, and the reverse inclusions are satisfied when $X(0)\in W_{-}$.
\item If $X(0)\in W_{+}\cup W_{-}$, then
\[
\begin{cases}
\gamma_{L}(u;X(0))\leq\gamma(u;X(0)) & \quad \forall u \in Dom(\gamma_{L}(\,\cdot\,;X(0)))\cap Dom(\gamma(\,\cdot\,;X(0)))\\
\gamma(u;X(0))\leq\gamma_{U}(u;X(0)) & \quad \forall u \in Dom(\gamma_{U}(\,\cdot\,;X(0)))\cap Dom(\gamma(\,\cdot\,;X(0)))
\end{cases}\,.
\] 
\item If $X(0)\in\left\{ \cF_{1}=0\right\} $ and $X_{1}(0)<u_{c}$ $($resp.
$X_{1}(0)>\bar{u}_{c})$ , then $\gamma,\gamma_{L},$ and $\gamma_{U}$
are still well-defined and items 2 and 3 still hold. 
\end{enumerate}
\end{corollary}
\begin{proof}
By \condI~ and the assumptions of items (1)-(3), we may apply \prettyref{lem:graph-comparison} by truncating $X$ at $\tau_{W_\pm^c}$,
from which items (1)-(3) follow. For item (4) we consider the case $X(0)<u_{c}$ and when $\bar u_c$ is finite, the case $X(0)>\bar u_c$ is analogous. In this case, by \prettyref{eq:Diff-ineq}, $\mathcal F_1 (X(0))=0$ and
\[
0<\cF_{2}^{L}(X(0))\leq\cF_{2}(X(0))\leq\cF_{2}^{U}(X(0))\,.
\]
Thus by continuity and the definitions of $\cF_{i}^{L/U}$ it follows
that $\cF_{1}(X(t)),\cF_{1}(A_{L}(t)),$ and $\cF_{1}(A_{U}(t))$ are all strictly
positive for $t>0$. Thus the conditions of \prettyref{lem:graph-comparison}
still hold. 
\end{proof}

\subsection{A pointwise comparison}
One might also be interested in obtaining pointwise-in-time comparisons
between the flows as in usual comparison theory. To this end we observe the following which
is not used in the proof of the phase diagram, but we believe is illustrative.
Define the region 
\begin{align*}
V_- = \{(u,v) \in\mathbb R_+^2: v\leq 2pu/\beta\}\,.
\end{align*}

\begin{figure}
\centering
\begin{tikzpicture}
\node at (0,0) {\includegraphics[width=0.30\textwidth]{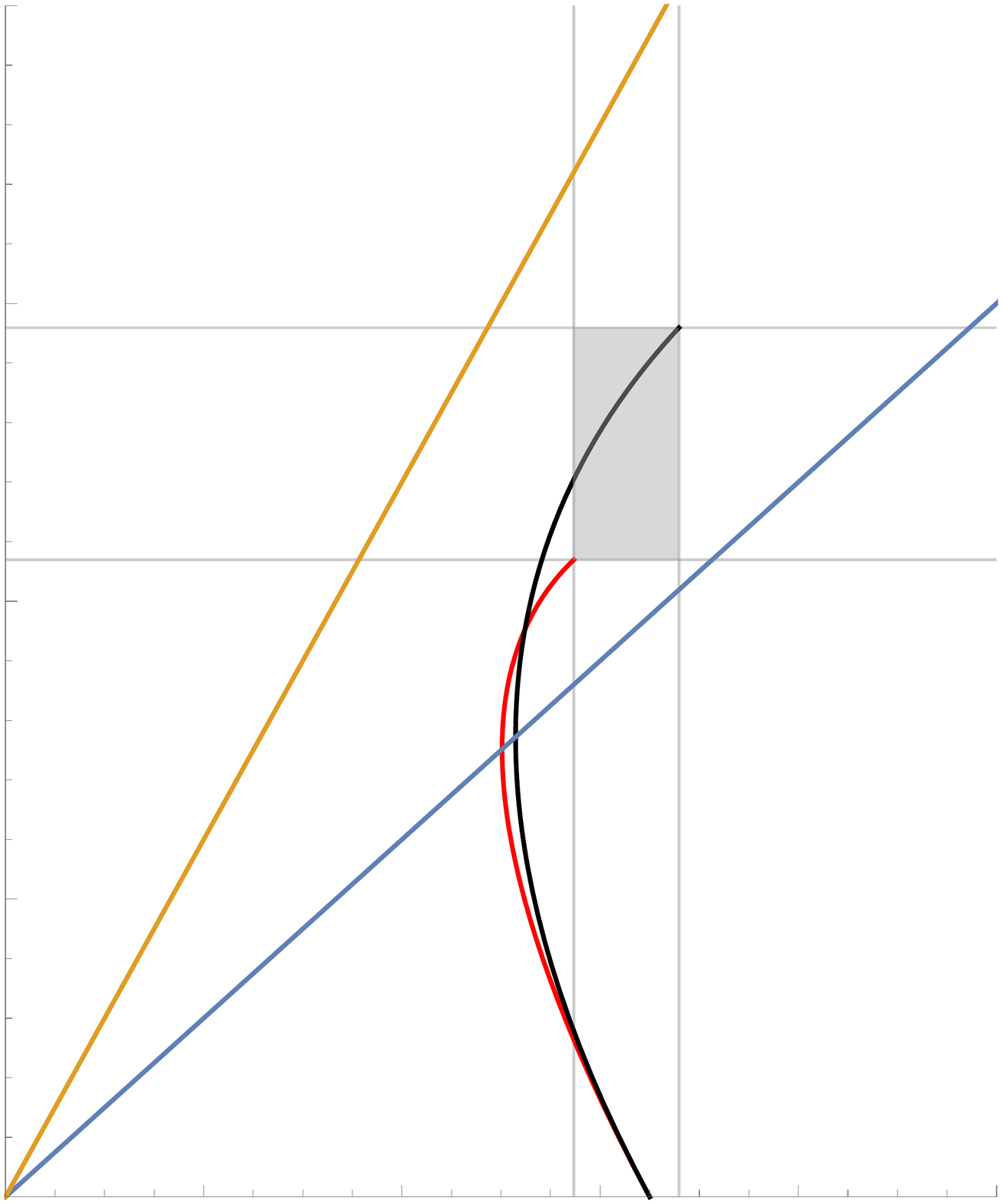}};
\node at (5.5,0) {\includegraphics[width=0.30\textwidth]{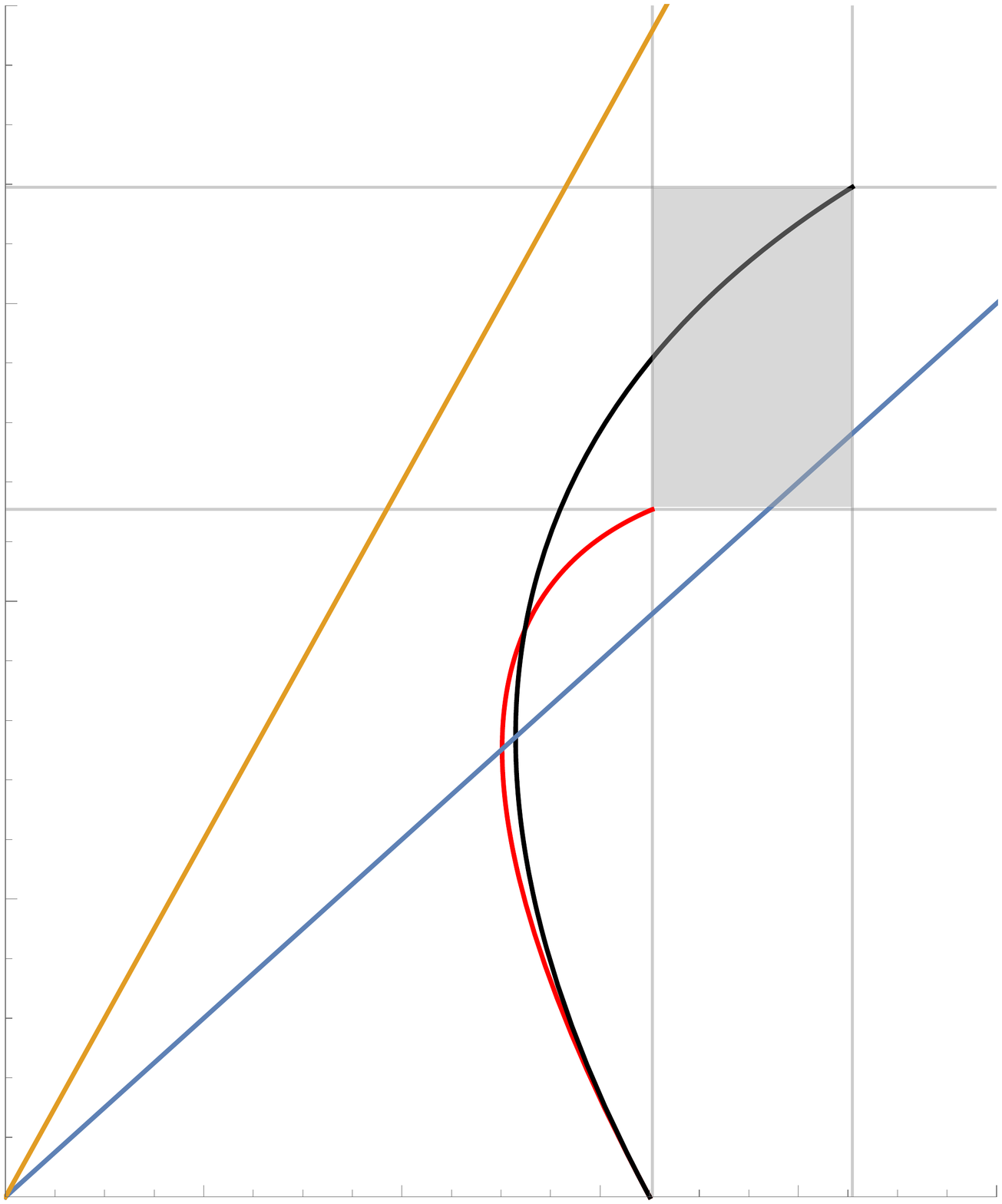}};
\node at (2,-2.4) {$\frac{H}{N}$};
\node at (-1.8,2.5) {$\frac{|\nabla H|^2}{N}$};
\node at (7.5,-2.4) {$\frac{H}{N}$};
\node at (5.5-1.8,2.5) {$\frac{|\nabla H|^2}{N}$};
\end{tikzpicture}
\caption{Theorem~\ref{thm:rectangle-condition} shows that $U(t)=(u(t),v(t))$ is $\mathscr L$-a.s.\ contained in the rectangle drawn out by $A_L(t)$ (red) and $A_U(t)$ (black) for all $t$ such that the rectangle is contained in $V_-$ (below the gold line).}\label{fig:trajectory-comparison}
\end{figure}

\begin{theorem}\label{thm:rectangle-condition}
Suppose that $X$ satisfies \condI. If $X(0)\in V_-^o$ then: 
\begin{enumerate}
\item  $X(t)\geq A_L(t)$ for every $t \leq \tau^L_{V_-^c} \wedge \tau_{V_-^c}$, and  $X(t) \leq A_U(t)$ for every $t \leq \tau^U_{V_-^c} \wedge \tau_{V_-^c}$.
\item  In particular, if $\tau^{\square}_B$ is the hitting time of $B$ for $(A_L^1(t),A_U^2(t))$, then for every $t \leq \tau^{\square}_{V_-^c}$.
\begin{align*}
A_L(t) \leq X(t) \leq A_U(t)\,.
\end{align*}
\end{enumerate}
As a result, if $X(0)$ is such that $\tau_{V_-^{c}}^\square= \infty$, then $A_L(t)\leq X(t)\leq A_U(t)$ for all $t\geq 0$.
\end{theorem} 

Observe the following, whose proof is immediate from the definition of $A_L$ and $A_U$.
\begin{observation}\label{lemma:interior-tau}
For every $(x_0,y_0)\in W^o_+$, if $X(0)=(x_0,y_0)$, we have that $\tau_{V_-^c}^\square >0$.
\end{observation}
The proof of the Theorem~\ref{thm:rectangle-condition} uses the classical comparison theorem of Chaplygin (see, e.g., \cite{Szarski65}). 

\begin{lemma}[Chaplygin's Lemma]\label{lem:chaplygin}
Let $\Omega\subset \R^2$ and let $A,B:[0,T]\to\Omega$ be continuously differentiable. 
Suppose that the following hold: 
\begin{enumerate}
\item $B(0)\leq A(0)$ $($resp.\ $B(0)\geq A(0)$$)$
\item There is a locally Lipschitz $\Psi:\R^{2}\to\R^{2}$ such that: $\dot{B}=\Psi(B)$
and $\dot{A}\geq\Psi(A(t))$ for every $t\leq T$ $($resp. $\dot{A}\leq\Psi(A(t))$$)$
\item  $\Psi$ is quasi-increasing: that is,  for each $x$, the map $\Psi_1(x, \cdot)$
is monotone increasing, and for every $y$,  the map $\Psi_2(\cdot, y)$ is
monotone increasing.
\end{enumerate}
Then $A(t)\geq B(t)$ \emph{(}resp. $B(t)\geq A(t)$\emph{)} for all $t\leq T$. 
\end{lemma}

\begin{proof}[\textbf{\emph{Proof of Theorem~\ref{thm:rectangle-condition}}}]
We check that $X(t)$ satisfies the conditions of Chaplygin's lemma. 
The first condition is satisfied by choice of initial data. 
The second follows by  \condI. 
Finally, by an explicit computation, it is easy to see that the region $V_-$ is the domain in which condition (3)  holds
for both $\Phi_L$ and $\Phi_U$. The result then follows by \prettyref{lem:chaplygin}.
\end{proof}

\section{Proof of full phase diagram}\label{sec:phase-portrait}

In this section, we turn to proving Theorem~\ref{thm:full-phase-portrait}.
One can of course obtain sharper
results for a given choice of initial data, however, we do not pursue
this direction. 

\subsection{Properties of the bounding flows}

We begin this section by first studying the phase diagram
for the two ``bounding flows''; we will then use the results we glean
to obtain corresponding results for the flow itself.

As in previous sections, let $A_L$ and $A_U$ be the lower and upper bounding flows respectively, and recall the definition of $W_+$ and $W_-$ from~\eqref{eq:W-pm} as well as $\bar u_c$ from~\eqref{eq:bar-u_c} and the decomposition of $\mathbb R\times \mathbb R_+$ in~\eqref{eq:A_0-eps}--\eqref{eq:A_i}.

\begin{lemma}
\label{lem:W+lem}We have the following:
\begin{enumerate}
\item If $(x_{0},y_{0})\in W_{+}$ then $A_{L}^{1}(\tau_{W_{+}^{c}}^{L})\geq u_{c}$
and likewise, if $\bar u_c <\infty$, then $A_{U}^{1}(\tau_{W_{+}^{c}}^{U})\geq\bar{u}_{c}$. 
\item Let $B$ be a bounded set. For every $\delta>0$ there exists $T_{0}$
such that for every $(x_{0},y_{0})\in B\cap W_{+}$, we have $\tau_{A_{0,\delta}\cup A_{1}}^{L}\leq T_{0}$
and if $\bar{u}_{c}<\infty$ then also $\tau_{A_{0,\delta}\cup A_{1}}^{U}\leq T_{0}$. 
\end{enumerate}
\end{lemma}
\begin{proof}
The fact that $A_{L}^{1}(\tau_{W_{+}^{c}}^{L})\geq u_{c}$ is shown
as follows. We first claim that $\tau_{\{u<u_{c}\}\cap W_{-}}^{L}\geq\tau_{\{u=u_{c}\}}^{L}$.
This follows from the fact that $\mathcal{F}_{2}^{L}(u,\ell_{1}(u))>0$ and $\cF_1(u,\ell_1(u))=0$
for all $u<u_{c}$, implying that $A_{L}^{1}(t)$ cannot hit $\{(u,\ell_{1}(u)):u<u_{c}\}$
for any $t\in(0,\tau_{\{u\geq u_{c}\}}^{L})$. 

Then since $\mathcal{F}_{1}(u,v)\geq0$ when $u\geq u_{c}$ and $v\geq f_{L}(u)$,
 $A_{L}^{1}(t)\geq u_{c}$ for all $t\in[\tau_{\{u\geq u_{c}\}\cap W_{+}}^{L},\tau_{W_{+}^{c}}^{L}]$,
implying the desired. Whenever $\bar{u}_{c}<\infty$ the analogous claim for $A_{U}^{1}(\tau_{W_{+}^{c}}^{U})$ follows similarly. 

It remains to prove the upper bounds on $\tau_{A_{0,\delta}}^{L}$
and $\tau_{A_{0,\delta}}^{U}$. We prove the bound for $A_{L}(t)$
as  the bound for $A_{U}(t)$ (when $\bar{u}_{c}<\infty$) follows
analogously. Fix $\delta>0$. By continuity, there exists an $\eta>0$ such that
$\ell_{1}(u)+\eta<f_{L}(u)$ for every $u\leq u_{c}-\delta$. Split
up $A_{4}\cup A_{2}$ into the following sets: 
\begin{align*}
E_{1}= & \{(u,v)\in(W_{+}\cap B)\setminus A_{0,\delta}:u\leq u_{c}-\delta,v\leq\ell_{1}(u)+\eta\}\,,\\
E_{2}= & \{(u,v)\in(W_{+}\cap B)\setminus A_{0,\delta}:u\leq u_{c}+\delta\}\setminus E_{1}\,,\\
E_{3}= & (W_{+}\cap B)\setminus(E_{1}\cup E_{2})\,.
\end{align*}

Observe that there exists a $c>0$ such that $\inf_{E_{1}}\mathcal{\mathcal{F}}_{2}^{L}\geq c$.
As a result, since $E_{1}$ is bounded from above by $E_{2}\cup A_{0,\delta}$, there
exists $T_{1}$ such that for every $(x_{0},y_{0})\in E_{1}$, $\tau_{E_{2}\cup A_{0,\delta}}^{L}\leq T_{1}$.
Observe that furthermore, $\inf_{E_{2}}\cF_{1}\geq c$, for some
other $c>0$. The continuity of $A_{L}(t)$ then implies that,
for every $(x_{0},y_{0})\in E_{2}$, $A_{L}$ cannot enter $E_{1}$
from $E_{2}$, which is to say that $\tau_{E_{2}}^{L}>\tau_{W_{+}^{c}}^{L}$
for $(x_{0},y_{0})\in E_{2}$. Moreover, since $E_{2}$ is a bounded
set and bounded to the right by $A_{0,\delta}\cup E_{3}$, there
exists a $T_{2}$ such that for every $(x_{0},y_{0})\in E_{2}$, $\tau_{A_{0,\delta}\cup E_{3}}^{L}\leq T_{2}$,
and, by continuity of $A_{L}(t)$, the flow cannot enter $E_{2}$
from $A_{0,\delta}\cup E_{3}$. Since $\sup_{E_{3}}\mathcal{F}_{2}^{L}\leq-c$,
and $E_{3}$ is bounded and bounded below by $A_{0,\delta}\cup A_{1}$,
there exists $T_{3}$ such that for every $(x_{0},y_{0})\in E_{3}$,
$A_{L}$ satisfies $\tau_{A_{0,\delta}\cup A_{1}}^{L}\leq T_{3}$.
Setting $T_{0}=T_{1}+T_{2}+T_{3}$ yields the desired estimate. 
\end{proof}

\begin{lemma}
\label{lem:A1-lem} If $\bar u_c<\infty$,  $A_{1}$ is nonempty, and we have the following.
\begin{enumerate}
\item If $(x_{0},y_{0})\in A_{1}$, then $A_{L}(\tau_{A_{1}^{c}}^{L})\in A_{0}$
and $A_{U}(\tau_{A_{1}^{c}}^{U})\in A_{0}$. 
\item Let $B$ be a bounded set. For every $\delta>0$, there exists $T_{0}$
such that for every $(x_{0},y_{0})\in B\cap A_{1}$, we have that
$\tau_{A_{0,\delta}}^{L}\leq T_{0}$ and that $\tau_{A_{0,\delta}}^{U}\leq T_{0}$. 
\end{enumerate}
\end{lemma}
\begin{proof}
Fix any $\delta>0$. First of all, since $f_{U}(u)$ intersects $\ell_{1}(u)$
at $\bar{z}_{c}\in A_{0}$, there exists a $c>0$ such that for
every $(u,\ell_{1}(u))$ with $u\geq\sup\{x:(x,y)\in A_{0,\delta}\}$
we have that $\mathcal{F}_{2}^{U}(u,\ell_{1}(u))\leq-c$. As
a consequence, we have both $A_{L}(\tau_{A_{1}^{c}}^{L})\notin A_{2}$
and $A_{U}(\tau_{A_{1}^{c}}^{U})\notin A_{2}$. Moreover, since $f_{L}(u)<f_{U}(u)$
on the domain of $f_{U}(u)$, there is some $c'>0$ such that
$\mathcal{F}_{2}^{U}(u,f_{L}(u))\geq c'$ for every $u\geq u_{c}+\delta$.
Thus $A_{U}(\tau_{A_{1}^{c}}^{U})\notin A_{3}$ so that $A_{U}(\tau_{A_{1}^{c}}^{U})\in A_{0}$. 

To deduce item 1 above, it remains to show that $A_{L}(\tau_{A_{1}^{c}}^{L})\notin A_{3}$.
By \prettyref{lem:graph-comparison} and the fact that $A_{L}(\tau_{A_{1}^{c}}^{L})\notin A_{2}$,
for every $(x_{0},y_{0})\in A_{1}$, the set $(A_{L}(t))_{t\leq\tau_{A_{1}^{c}}}$
is the graph of some $\gamma_{L}(u;(x_0,y_0))$ with 
\[
\gamma'_{L}(u)=-\frac{\mathcal{F}_{2}^{L}(u,\gamma_{L}(u))}{\mathcal{F}_{1}(u,\gamma_{L}(u))}\,.
\]
 This implies that if $A_{L}(t_{0})=(u_{0},f_{L}(u_{0}))$ for some
$u_{0}>u_{c}$, then $\gamma'_{L}(u_{0})=0$ while $f_{L}'(u_{0})>0$.
As a result, $\gamma_{L}(u)\in A_{1}$ for all $u<u_{0}$ sufficiently
close to $u_{0}$. Consequently, $A_{L}(t)\in A_{1}$ for all $t>t_{0}$
sufficiently close to $t_{0}$. By continuity of $A_{L}(t)$, it then
follows that $A_{L}(\tau_{A_{1}^{c}}^{L})\notin A_{3}$ as desired. 

To see item 2, notice that $\inf_{A_{1}\setminus A_{0,\delta}}\abs{(\mathcal{F}_{1},\mathcal{F}_{2}^{L})}\geq c$
for some $c>0$ and, moreover, $\mathcal{F}_{1}(u,v)\leq0$
and $\mathcal{F}_{2}(u,v)\leq0$ for every $(u,v)\in A_{1}$. Since
$A_{1}$ is bounded to the left, there exists a $T_{0}$ such that
for every $(x_{0},y_{0})\in A_{1}$, we have that $\tau_{(A_{1}\setminus A_{0,\delta})^{c}}^{L}\leq T_{0}$.
By item 1, this implies that $\tau_{A_{0,\delta}}^{L}\leq T_{0}$.
Item~2 for the upper bounding flow $A_{U}(t)$ follows similarly,
after splitting up $A_{1}$ into those points above $\{u,f_{U}(u)\}$
and those below $\{u,f_{U}(u)\}$ and noting that $A_{U}(t)$ cannot
cross $f_{U}(u)$ from above. 
\end{proof}

\subsection{Proof of Theorem~\ref{thm:full-phase-portrait}}\label{subsec:Phase-portrait-for-X}
In the following, we will be interested in phase portraits for systems that satisfy \condI~
and \condB~ as defined in Definitions~\ref{def:condi}--\ref{def:condb}.
Thus, for the rest of this section, we assume that $U$ satisfies both \condI~
and \condB~with respect to some $T$ and $W$. In particular, we will take this $T$ to be sufficiently large such that all finite hitting times we consider are less than $T$. 
As before, let $A_{i}$ denote the
sets from~\eqref{eq:A_0-eps}--\eqref{eq:A_i} restricted to the compact set $W$ for which
$U$ satisfies \condB~. Recall that for any set $A\subset W$, we let $\tau_{A}$ denote
the hitting time for $U$ started from $U(0)=(x_0,y_0)$. 

\begin{proof}[\textbf{\emph{Proof of item~\eqref{item:A4} of Theorem~\ref{thm:full-phase-portrait}}}]
We wish to show that for every $(x_0 , y_0)\in A_{4}$, we have $U({\tau_{A_{4}^c}})\in A_2$ and $\tau_{A_{4}^c}\leq T_0$.
The fact that $U(\tau_{A_{4}^c}) \in A_2$ is trivial by definition of $A_{4}$ and $A_2$ (as $\partial A_{4} \subset A_2$). 
By item (2) of Corollary~\ref{cor:graph-cor} and Lemma~\ref{lem:W+lem}, for every $x_0 <0$ and every  $u\in [x_0, 0]$, for $T$ large enough, $u$ is in the domain of $\gamma_L(\cdot ;(x_0 ,y_0))$ and 
\begin{align*}
\ell_1(u)<\gamma_L(u;(x_0,y_0)) \leq \gamma(u;(x_0 ,y_0))\,.
\end{align*}
By Lemma~\ref{lem:W+lem}, there exists $T_0$ such that for every $(x_0,y_0)\in A_{4} \subset W_+$, we have $\tau_{A_2}^L \leq T_0$. This implies by Corollary~\ref{cor:graph-cor} that $(U(t))_{t\leq \tau_{A_2}}$ is the graph of a function $\gamma(u;(x_0,y_0))$ whose domain contains $[x_0, 0]$. 
For every $t\leq \tau_{A_{4}^c}$, we have $|\dot U(t)|\leq c>0$ for $c(p,\beta;(x_0,y_0))$ depending continuously on $(x_0,y_0)$. Together these imply that for every $(x_0,y_0)\in A_{4}$, there exists $T_{x_0,y_0} (p,\beta)$ depending continuously on $(x_0,y_0)$ such that $\tau_{A_{4}^c}\leq T_{x_0,y_0}$. As $T_{x_0,y_0}$ is continuous and $W$ is compact, the supremum over all $(x_0,y_0) \in A_{4}$ of $T_{x_0,y_0}$ is some finite $T_0$ as desired. 
\end{proof}


\begin{proof}[\textbf{\emph{Proof of item~\eqref{item:A3} of Theorem~\ref{thm:full-phase-portrait}}}]
Tautologically, for every $(x_0,y_0)\in A_3$, we have $U(\tau_{A_3^c})\in A_0\cup A_1 \cup A_2$; it remains to show that $\tau_{A_{0,\delta} \cup A_1 \cup A_2}\leq T_0$.  Fix any $\delta$ small and observe that
for every $(u,v)\in A_{3}\backslash A_{0,\delta}$,
\[
\cF_{1}(u,v)\leq0\qquad\mbox{and}\qquad\cF_{2}^{L}(u,v)\geq0\,.
\]
Moreover, there is a $c(\delta)>0$ such that for every $(u,v)\in A_{3}\setminus A_{0,\delta}$,
we have $\abs{\Phi_{L}}>c$ and, for every  $(x_0,y_0)\in A_3 \setminus A_{0,\delta}$, for every $t\leq \tau_{(A_3\setminus A_{0,\delta})^c}$, we have  $|{\dot{U}}(t)|>c$. As a result, by compactness
of $A_{3}\backslash A_{0,\delta}$ and compactness of $W$, there exists $T_{0}$ such that for every $(x_{0},y_{0})\in A_{3}\setminus A_{0,\delta}$,
\[
\tau_{\left(A_{3}\backslash A_{0,\delta}\right)^{c}}=\tau_{A_{0,\delta}\cup A_1 \cup A_{2}}\leq c^{-1}\sup_{\substack{w\in A_{3}}
,w'\in\partial\left(A_{3}\backslash A_{0,\delta}\right)}d(w,w')\leq T_{0}\,. \qedhere
\]
\end{proof}

\begin{proof}[\textbf{\emph{Proof of item~\eqref{item:A2} of Theorem~\ref{thm:full-phase-portrait}}}]
We first show that $U(\tau_{A_2^c})\in A_0\cup A_1$. 
The fact that $U(\tau_{A_{2}^{c}})\notin A_{4}$ follows because $\mathcal{F}_{1}(u,v)\geq0$
for all $(u,v)\in W_{+}$ . By item (2) of \prettyref{cor:graph-cor}
combined with \prettyref{lem:W+lem}, we see that for every $u\in [x_0,u_{c})$,
 if $T$ is large enough, $u$ is in the domain of $\gamma_{L}(\cdot;(x_{0},y_{0}))$
and
\begin{align*}
\ell_{1}(u)<\gamma_{L}(u;(x_{0},y_{0}))\leq & \gamma(u;(x_{0},y_{0}))\,;
\end{align*}
this implies that $U(\tau_{A_{2}^{c}})\notin A_{3}$ implying $U(\tau_{A_{2}^{c}})\in A_{0}\cup A_{1}$. 

It remains to show that for every $\delta>0$, there is a $T_0$ such that $\tau_{A_{0,\delta}\cup A_1}\leq T_0$. 
Recall that by \prettyref{lem:W+lem}, for every $\delta>0$, 
there exists $T_{0}(\delta)<\infty$ such that for every $(x_{0},y_{0})$,
$\tau_{A_{0,\delta}\cup A_{1}}^{L}\leq T_{0}$. This implies by items (1) and (2) of \prettyref{cor:graph-cor}, $(U(t))_{t\leq\tau_{A_{0,\delta}\cup A_{1}}}$
is the graph of a function $\gamma(u;(x_{0},y_{0}))$ with domain containing
$[x_{0},U_1(\tau_{A_{0,\delta}\cup A_{1}})]$. 

Now notice that for every $\delta>0$ and every  $t\leq\tau_{A_{0,\delta}\cup A_{1}}$,
we have $|{\dot{U}(t)}|\geq c>0$ for some $c(p,\beta;(x_{0},y_{0}))$
that depends continuously on $(x_{0},y_{0})$. As the boundary of
$A_{2}$ relative to $\{u\geq x_0, v\geq \gamma_L(u)\}$ is $A_{0,\delta}\cup A_{1}$,
we have that for every $(x_{0},y_{0})\in  A_{2}$ there exists
$T_{x_{0},y_{0}}(p,\beta)$ depending continuously on $(x_{0},y_{0})$
such that $\tau_{A_{0,\delta}\cup A_{1}}\leq T_{x_{0},y_{0}}$.
As $T_{x_{0},y_{0}}$ is continuous and $W$ is compact, we may bound
$T_{x_{0},y_{0}}$ uniformly over all such $(x_{0},y_{0})$ by some
finite $T_{0}$, as desired.
\end{proof}

\begin{proof}[\textbf{\emph{Proof of item~\eqref{item:A1} of Theorem~\ref{thm:full-phase-portrait}}}]
By \prettyref{lem:A1-lem},  $A_{L}(\tau_{A_{1}^{c}}^{L}),A_{U}(\tau_{A_{1}^{c}}^{U})\in A_{0}$
and, for every $\delta>0$, there is a $T_{1}$ such that
$\tau_{A_{0,\delta}}^{L}\leq T_{1}$ and $\tau_{A_{0,\delta}}^{U}\leq T_{1}$.
Thus by \prettyref{cor:graph-cor}, $U(\tau_{A_{1}^{c}})\in A_{0}$. 

We now show that for every $\delta>0$ there is a $T_0$ such that $\tau_{A_{0,\delta}}\leq T_0$. We begin with the following velocity
estimates. By \eqref{eq:Diff-ineq}, there exists $c>0$ such
that $|{\dot{U}}(t)|\geq c$ for every $t\leq\tau_{(A_{1}\setminus A_{0,\delta})^{c}}$.
We now split the region $A_{1}\backslash A_{0,\delta}$ into two
regions as follows. For starting points $(x_{0},y_{0})$ satisfying
$x_{0}>\bar{u}_{c}+\delta$ and $y_{0}=\ell_{1}(x_{0})-\eta$, as
$\eta\to0,$ we have that $\gamma'_{U}(x_{0};(x_{0},y_{0}))\to\infty$.
Since $\gamma_{U}^{'}(x_{0};(x_{0},y_{0}))$ depends continuously
on $(x_{0},y_{0})$, there exists an $\eta(\delta)>0$ small such
that for every $(u,v)$ with $u\geq\bar{u}_{c}+\delta$ and
$v=\ell_{1}(u)-\eta$, we have $\gamma'_{U}(u;(u,v))>p/\beta$. With
respect to this $\eta$, define the sets
\begin{align*}
E_{+} & =\left\{ (u,v)\in A_{1}\setminus A_{0,\delta}:y\geq\ell_{1}(u)-\eta\right\}\,, \qquad \mbox{and}\qquad E_{-} =(A_{1}\setminus A_{0,\delta})\setminus E_{+}\,.
\end{align*}

For every  $(u,v)\in E_{+}$, we have that $\cF_{1}(u,v)\leq0$ and $\cF_{2}^{U}(u,v)\leq0$,
so that by \eqref{eq:Diff-ineq},
\[
\tau_{E_{-}\cup A_{0,\delta}}=\tau_{E_{+}^{c}}\leq2c^{-1}\sup_{\substack{w\in E_{+},w'\in\partial A_{1}}
}d_{\ell_{1}}(w,w')=T_{2}\,,
\]
for some finite $T_{2}(\delta)$, uniformly over $(x_{0},y_{0})\in E_{+}$.
On the other hand, since integral curves of locally smooth flows cannot
cross and since the flow $A_{U}$ is outward along the boundary of
$E_{+}$ relative to $A_{1}\backslash A_{0,\delta}$ (as $\gamma_{U}'(u;(u,v))>p/\beta$
and $\mathcal{F}_{1}\leq0$), it follows that $\tau_{E_{-}^{c}}^{U}=\tau_{A_{0,\delta}}^{U}$
for every $(x_{0},y_{0})\in E_{-}$. As a result, by \prettyref{cor:graph-cor}
and we also see that $\tau_{E_{-}^{c}}=\tau_{A_{0,\delta}}$. Moreover,
there is a $c>0$ such that $\inf_{E_{-}}\mathcal{F}_{1}\leq-c$,
so by compactness of $W$ and the fact that $E_{-}$ is bounded to
the left by $A_{0,\delta}$, there is a $T_{3}(\delta)$ such
that for every $(x_{0},y_{0})\in E_{-}$, we have $\tau_{A_{0,\delta}}=\tau_{E_{-}^{c}}\leq T_{3}$.
Thus for every $(x_{0},y_{0})\in A_{1}$ we have that $\tau_{A_{0,\delta}}\leq T_{2}+T_{3}$
as desired. 
\end{proof}

\begin{proof}[\textbf{\emph{Proof of item~\eqref{item:absorbing} of Theorem~\ref{thm:full-phase-portrait}}}]
It remains, for each choice of $\epsilon>0$, to construct an absorbing set $\mathcal A_\epsilon$ containing $A_{0,\delta}$ for some $\delta(\epsilon)>0$. Notice that while the above indicates that the hitting time to $A_{0,\delta}$ is finite, $\tau_{A_{0,\delta}^{c}}$ started from $(x_{0},y_{0})\in A_{0,\delta}$
is not necessarily $\infty$. 
Let $z(\epsilon)=(u_{c}-\epsilon,\ell_{1}(u_{c}-\epsilon))$
and $\bar{z}(\epsilon)=(\bar{u}_{c}+\epsilon,\ell_{1}(\bar{u}_{c}+\epsilon))$
so that $z_{c}=z(0)$ and $\bar{z}_{c}=\bar{z}(0)$:
\begin{align*}
A_{\epsilon} & =\begin{cases}
\left\{ (u,v)\in W:v\in[\ell_{1}(u)\vee\gamma_{L}\left(u;\bar{z}(\epsilon)\vee A_{U}(\tau_{W_{+}^{c}}^{U};z(\epsilon))\right),\gamma_{U}(u;z(\epsilon))]\right\}  & \mbox{if }\bar{u}_{c}<\infty\\
\left\{ (u,v)\in W:v\in[f_{L}(u),\gamma_{U}(u;z(\epsilon))]\right\}  & \mbox{otherwise}
\end{cases}\,.
\end{align*}
where again, the $u$ values are implicitly constrained by the domain
of $\gamma_{U}(u;z(\epsilon))$. Let 
\[
B_{\epsilon}(z_{c})=([u_{c}-\epsilon,u_{c}+\epsilon]\times[\ell_{1}(u_{c}-\epsilon),\ell_{1}(u_{c}+\epsilon)])\cap W_{-}\,.
\]
and define $B_{\epsilon}(\bar{z}_{c})$ analogously but intersected
with $W_{+}$ (again if $\bar{u}_{c}=\infty$, this is empty). Let
\begin{align}\label{eq:mathcal-A-eps}
\mathcal{A}_{\epsilon}=A_{\epsilon}\cup B_{\epsilon}(z_{c})\cup B_{\epsilon}(\bar{z}_{c})\,.
\end{align}
It follows by the definitions
and the continuity of $\Phi_{L}$ and $\Phi_{U}$ that for every $\epsilon>0$, there is a $\delta>0$
such that $A_{0,\delta}\subset\mathcal{A}_{\epsilon}$.
We now turn to showing that for each $\epsilon>0$, $\cA_{\epsilon}$ is an absorbing set.

By continuity of $U$, it suffices to show that if for any $t_0\geq 0$, $U(t_{0})\in\partial\cA_{\epsilon}$
then for all $t>t_{0}$ sufficiently close to $t_{0}$, $U(t)\in\cA_{\epsilon}$.
We decompose the boundary into its constituent parts and analyze each
part. We consider the case $\bar{u}_{c}<\infty.$

As $\bar{u}_{c}<\infty$, notice that by definition, 
\begin{align*}
\partial\cA_{\epsilon} & \subset \Gamma(\gamma_{L})\cup\Gamma(\gamma_{U})\cup\partial B_{\epsilon}(z_{c})\cup\partial B_{\epsilon}(\bar{z}_{c})\,,
\end{align*}
 where $\Gamma(\gamma_{L})$ is the graph of $\gamma_{L}\left(u;\bar{z}(\eps)\vee A_{U}(\tau_{W_{+}^{c}}^{U};z(\eps))\right)$
and $\Gamma(\gamma_{U})$ is the graph of $\gamma_{U}(u;z(\eps))$. 

First, suppose that $U(t_{0})\in\Gamma(\gamma_{U})$. Since this is
in the region $W_{+}$, 
\[
\gamma(w;U(t_{0}))\leq\gamma_{U}(w;U(t_{0}))\,,
\]
for $w> u$ sufficiently close to $u$ by \prettyref{cor:graph-cor}.
For such $w$, $\gamma_{U}(w;z(\eps))=\gamma_{U}(w;U(t_{0}))$ as these are integral curves, so it follows from \prettyref{cor:graph-cor} that for $t$ sufficiently close to $t_{0}$, $U(t)\in\cA_{\epsilon}$
as desired. The case when $U(t_{0})\in\Gamma(\gamma_{L})$ of course
follows analogously.

Now suppose that $U(t_{0})\in\partial B_{\epsilon}(z_{c}).$ Observe
that if $\hat{n}_{z}$ denotes the outward pointing normal at $z\in\partial B_{\epsilon}(z_{c})$,
then for $z$ in the right face and the bottom face we have that
$\dot{U}(t_{0})\cdot\hat{n}_{U(t_{0})}<0$, which yields the
desired statement. The remaining face is in the interior of $\cA_{\epsilon}$
so we can ignore this case. The case when $U(t_{0})\in\partial B_{\epsilon}(\bar{z}_{c})$
also follows analogously.

In the situation when $\bar{u}_{c}=\infty$, $\Gamma(\gamma_{L})$ is replaced
by $\Gamma(f_{L})=\{(u,v):u>u_{c},v=f_{L}(u)\}$ and the corresponding
estimate follows by recalling from the proof of item~\eqref{item:A1} of Theorem~\ref{thm:full-phase-portrait} that
if $(x_{0},y_{0})\in\Gamma(f_{L})$, then $A_{L}(t)\in A_{0}\cup A_{1}$
for all sufficiently small $t$. The desired then again follows by
\prettyref{cor:graph-cor}.
\end{proof}

\appendix
\section{Proof of Lemma~\ref{lem:diff-geo-fact}}\label{app:diff-geo}
In this section, we provide a proof of Lemma~\ref{lem:diff-geo-fact} which we had deferred.  
\begin{proof}[\textbf{\emph{Proof of Lemma~\ref{lem:diff-geo-fact}}}]
The proof is by induction. As a base case, observe that 
\begin{align*}
\nabla^{2}f(X_{1},X_{2}) & =\hat{\nabla}^{2}f(X_{1},X_{2})-\frac{1}{\sqrt{N}}\hat{\nabla}f(P)\left\langle X_{1},X_{2}\right\rangle \\
 & =\hat{\nabla}^{2}f(X_{1},X_{2})+\frac{1}{\sqrt{N}}\sum_{\sigma\in S_{3}}c_{1,2}(\sigma)\hat{\nabla}f(W_{\sigma(1)}^{1,2})\left\langle W_{\sigma(2)}^{1,2},W_{\sigma(3)}^{1,2}\right\rangle\,,
\end{align*}
for some suitably defined $c_{1,2}$. Suppose now that~\eqref{eq:kth-deriv} holds for some $k\geq 2$. Then, by the inductive hypothesis, to compute $\nabla^{k+1}f$
it suffices to compute $\nabla_{Y}A_{\ell,k}(\sigma)$ where $Y=X_{k+1}$.

First notice that $A_{\ell,k}$ is a product of the form
\[
\frac{1}{N^{(k-\ell)/2}}\hat{\nabla}^{\ell}f\tensor Id\tensor\ldots\tensor Id\,.
\]
As the covariant
derivative of the metric tensor is zero, we see by the product rule,
it suffices to compute $V=\nabla_{Y}\hat{\nabla}^{\ell}f(W_{\sigma(1)}^{\ell,k},\ldots,W_{\sigma(\ell)}^{\ell,k}).$
From now on we suppress the dependence of $W$ on $k,\ell$ for readability.
Let $I=\left\{ i:\sigma(i)\leq k\right\} $. Then $V$ is the covariant
derivative of an $\abs{I}-$form. In particular, 
\begin{align*}
V & =Y\hat{\nabla}^{\ell}f(W_{\sigma(1)},\ldots,W_{\sigma(\ell)})-\sum_{s\in I}\hat{\nabla}^{\ell}f(W_{\sigma(1)},\ldots,\nabla_{Y}W_{\sigma(s)},\ldots,W_{\sigma(\ell)})\\
 & =\hat{\nabla}^{\ell+1}f(W_{\sigma(1)},\ldots,W_{\sigma(\ell)},Y)+\sum_{s\in I}\hat{\nabla}^{\ell}f(W_{\sigma(1)},\ldots,II(Y,W_{\sigma(s)}),\ldots,W_{\sigma(\ell)})\\
 & \qquad+\sum_{s\in I^{c}}\hat{\nabla}^{\ell}f(W_{\sigma(1)},\ldots,\hat{\nabla}_{Y}W_{\sigma(s)},\ldots W_{\sigma(\ell)})\,.
\end{align*}
where the second equality comes from the definition of the covariant
derivative and the Gauss formula. Since the shape operator of the sphere is multiplication
by $(-1/\sqrt{N})$, we see for $s\in I^{c}$, 
\[
\hat{\nabla}_{Y}W_{\sigma(s)}=\hat{\nabla}_{Y}P=\frac{1}{\sqrt{N}}Y=\frac{1}{\sqrt{N}}Y\left\langle P,P\right\rangle .
\]
by the Weingarten equation.
Plugging this, $Y=X_{k+1}$, and \eqref{eq:second-fund-form} in the above, we see that 
\[
V=\hat{\nabla}^{\ell+1}f(W_{\rho(1)}^{\ell+1,k+1},\ldots,W_{\rho(\ell+1)}^{\ell+1,k+1})+\frac{1}{\sqrt{N}}\sum_{\rho'\in E}\hat{\nabla}^{\ell}f(W_{\rho'(1)}^{\ell,k+1},\ldots,W_{\rho'(\ell)}^{\ell,k+1})\cdot\left\langle W_{\rho'(\ell+1)}^{\ell,k+1},W_{\rho'(\ell+1)}^{\ell,k+1}\right\rangle ,
\]
where $\rho\in Sym_{2(k+1)-(\ell+1)}$ is such that $\rho(\ell+1)=k+1$ and
$E\subset Sym_{2(k+1)-\ell}$.
Observe that in the first term, we have increased the indicies $k$
and $\ell$ and in the second term we have gained a factor of $N^{-1/2}$
and an inner-product. Thus 
\[
\nabla_{X_{k+1}}A_{\ell,k}(\sigma)=A_{\ell+1,k+1}(\rho)+\sum_{\rho'\in E}A_{\ell,k+1}(\rho')\,,
\]
for some $\rho$ and where the sum is again over some subset of $Sym_{2(k+1)-\ell}.$
Note that $\rho$ and this subset $E$ depend
only on $\sigma,\ell,$ and $k$. 
Thus, by the inductive hypothesis, 
\begin{align*}
\nabla^{k+1}f(X_{1},\ldots,X_{k+1}) & =(\nabla_{X_{k+1}}\nabla^{k}f)(X_{1},\ldots,X_{k})=\sum_{\ell=1}^{k}\sum_{\sigma\in Sym_{2k-\ell}}c_{\ell,k}(\sigma)\nabla_{X_{k+1}}A_{\ell,k}(\sigma)\\
 & =\sum_{\ell=1}^{k}\sum_{\sigma\in Sym_{2k-\ell}}c_{\ell,k}(\sigma)\left(A_{\ell+1,k+1}(\rho(\sigma))+\sum_{\rho'\in E(\sigma)}A_{\ell,k+1}(\rho')\right)\\
 & =\sum_{\ell=1}^{k+1}\sum_{\sigma\in Sym_{2(k+1)-\ell}}c_{\ell,k+1}(\sigma)A_{\ell,k+1}(\sigma)\,,
\end{align*}
for some $c_{\ell,k+1}$, concluding the proof.
\end{proof}

\bibliographystyle{plain}
\bibliography{boundingflows}

\end{document}